\documentclass{amsart}

\usepackage{amssymb}
\usepackage{amsthm}
\usepackage{amsmath}
\usepackage{graphicx}
\usepackage{hyperref}
\usepackage{multirow}
\usepackage{array}
\usepackage[colorinlistoftodos]{todonotes}
\usepackage[english]{babel}
\usepackage{lmodern}
\usepackage{comment}
\usepackage{bbold}
\usepackage{mathtools}
\usepackage{enumerate}
\usepackage{cite}
\usepackage[left=3cm,right=3cm,top=2.5cm,bottom=2.5cm]{geometry}

\allowdisplaybreaks

\usepackage[utf8]{inputenc}
\usepackage[T1]{fontenc}

\newtheorem{theorem}{Theorem}[section]
\newtheorem{lemma}[theorem]{Lemma}
\newtheorem{corollary}[theorem]{Corollary}

\newtheorem{proposition}[theorem]{Proposition}

\theoremstyle{definition}
\newtheorem{definition}[theorem]{Definition}
\newtheorem{notation}[theorem]{Notation}

\newtheorem{question}[theorem]{Question}
\newtheorem{example}[theorem]{Example}

\theoremstyle{remark}
\newtheorem{remark}[theorem]{Remark}
\newcommand{\longcomment}[1]{}

\DeclareMathOperator{\Div}{Div}
\DeclareMathOperator{\charac}{char}
\DeclareMathOperator{\ord}{ord}

\DeclareMathOperator{\Frob}{Frob}

\DeclareMathOperator{\Aut}{Aut}

\DeclareMathOperator{\Br}{Br}

\newcommand{\dm}[1]{D_{#1P}}
\newcommand{\dn}{\dm{n}}
\newcommand{\dmpr}[1]{D_{#1P'}}
\newcommand{\dnpr}{\dmpr{n}}

\newcommand{\cstE}{{\widetilde{E}}}
\newcommand{\cstEF}{{\widetilde{E}_{F}}}
\newcommand{\cstEFpr}{{\widetilde{E}_{F'}}}

\newcommand{\whichbold}[1]{\mathbf{#1}} 

\newcommand{\F}{\whichbold{F}}
\renewcommand{\P}{\whichbold{P}}
\renewcommand{\div}{\mathrm{div}}
\newcommand{\ZZ}{\whichbold{Z}}
\newcommand{\FF}{\whichbold{F}}

\renewcommand{\AA}{\whichbold{A}}
\newcommand{\QQ}{\whichbold{Q}}

\numberwithin{equation}{section}
\author{Bartosz Naskręcki}
\address{Faculty of Mathematics and Computer Science,
	Adam Mickiewicz University in Poznań,
	Uniwersytetu Poznańskiego~4,
	61-614 Poznań, 
	Poland}
\email{bartnas@amu.edu.pl}
\author{Marco Streng}
\address{Mathematisch Instituut,
	Universiteit Leiden,
	P.O.\ box 9512,
	2300 RA Leiden,
	The Netherlands
}
\email{streng@math.leidenuniv.nl}


\title[Elliptic divisibility sequences with constant $j$-invariant]{Primitive divisors of elliptic divisibility sequences over
	function fields with constant $j$-invariant}
\begin{document}
	\begin{abstract}
	We prove an optimal Zsigmondy bound for elliptic divisibility sequences
	over function fields in case the $j$-invariant of the elliptic curve
	is constant.
	
	In more detail, given an elliptic curve $E$ with a point $P$
	of infinite order over a global field, the sequence $D_1$, $D_2, \ldots$
	of denominators of
	multiples $P$, $2P,\ldots$ of $P$ is a strong divisibility sequence
	in the sense that $\gcd(D_m, D_n) = D_{\gcd(m,n)}$.
	This is the genus-one analogue of the genus-zero
	Fibonacci, Lucas and Lehmer sequences.
	
	A number $N$ is called a Zsigmondy bound of the sequence
	if each term $D_{n}$ with $n>N$ presents a new prime factor.
	The optimal uniform Zsigmondy bound
	for the genus-zero sequences over $\QQ$ is $30$ by Bilu-Hanrot-Voutier, 2000\nocite{Bilu_Hanrot_Voutier},
	but finding such a bound remains an open problem
	in genus one, both over $\QQ$ and over function fields.
	
	We prove that the optimal Zsigmondy bound for ordinary
	elliptic divisibility sequences over function fields is~$2$
	if the $j$-invariant is constant.
	In the supersingular case, we give a complete classification
	of which terms can and cannot have a new prime factor.
	\end{abstract}
\maketitle

\tableofcontents
\section{Introduction}
An \emph{elliptic divisibility sequence (EDS)} over $\QQ$
is a sequence $D_1$, $D_2$, $D_3,\ldots$
of positive integers defined as follows.
Given an elliptic curve $E$ over $\QQ$ and a point $P\in E(\QQ)$
of infinite order, choose a globally minimal Weierstrass
equation for $E$ and
write for every $Q\in E(\QQ)$:
\begin{equation}\label{eq:classicaleds}
Q = \left( \frac{A_Q}{D_Q^2}, \frac{C_Q}{D_Q^3}\right),
\end{equation}
where the fractions are in lowest terms.
Then set $D_n = \dn$.

A result of Silverman \cite{Silverman_Wieferich} shows that all but finitely many terms $D_n$
have a \emph{primitive divisor}, that is, a prime divisor $p\mid D_n$
such that $p\nmid D_m$ for all $1\leq m < n$.
Equivalently, this says that all but finitely many positive integers $n$
occur as the order of $(P~\mathrm{mod}~p)$ for some prime~$p$.
The question whether there is a uniform bound $N$ such that $D_n$
has a primitive divisor for all pairs $(E,P)$ and all $n>N$
remains open, see \cite{Cheon_Hahn},
\cite{Everest_Mclaren_Ward_EDS}, \cite{Ingram_EDS_over_curves}.

The definition of $D_Q$ of \eqref{eq:classicaleds} is equivalent to
\begin{equation}\label{eq:defedseasy}
v(D_Q) = \mathrm{max}\{ -\frac{1}{2}v(x_v(Q)), 0\}
\end{equation}
for all non-archimedean valuations $v$ and $x_v$ the
$x$-coordinate function for a $v$-minimal Weierstrass equation.
If $E$ and $P$ are defined over a number field~$F$,
then we define the \emph{EDS} of the pair $(E,P)$ to be the sequence
of ideals $D_n=D_{nP}$ of $\mathcal{O}_F$ defined by \eqref{eq:defedseasy}.

Similarly, if $E$ and $P$ are defined over the function field
$F=K(C)$ of a smooth, projective, geometrically irreducible curve $C$
over a field~$K$,
then we define the \emph{EDS} of the pair $(E,P)$ to be the sequence
of divisors $D_n=D_{nP}$ on~$C$ defined by~\eqref{eq:defedseasy}.
See Section~\ref{sec:alternativedef}
for an equivalent definition
in the case of perfect~$K$.
Elliptic divisibility sequences over function fields
are studied in \cite{Everest_Ingram_Mahe_Stevens_Uniform_EDS,
	Silverman_Common_divisors,
	IMSSS,
	Cornelissen_Reynolds}.

From now on, we will speak of \emph{primitive valuations}
instead of primitive divisors, so as not to confuse with the terms
themselves, which are divisors in the function field case.
A positive integer $N$ is a \textit{Zsigmondy bound} of
the sequence $(D_n)_{n}$ if for every $n>N$
the term $D_n$ has a primitive valuation.

Silverman's result and proof are also valid in the number field case \cite{Ingram_Silverman_uniform}.
In the case of function fields of characteristic zero,
the same result is true, as shown by
Ingram, Mah\'e, Silverman, Stange and
Streng~\cite[Theorems 1.7 and~5.5]{IMSSS}.

This was extended to ordinary elliptic curves $E$
over function fields of characteristic $\not=2,3$ by
Naskręcki~\cite{Naskrecki_NewYork}.
Conditionally Naskręcki makes the result uniform in~$E$.
The special case of the results of \cite{Naskrecki_NewYork}
where $j(E)$ is constant gives a Zsigmondy bound $N$ as follows.
\begin{itemize}
	\item For fields $K(C)$ of characteristic $0$ we have $N\leq 72$ (see 
	\cite[p. 437]{Hindry_Silverman} and \cite[Lemma~7.1]{Naskrecki_NewYork}).
	\item For fields $K(C)$ with $p=\charac K(C)\geq 5$ and field of constants $K=\F_{q}$, $q=p^s$ we have
	$N
	<10^{100(15+20g(C))}\cdot p^{84}$ for `tame' elliptic curves
	(cf.~\cite[Definition~8.3]{Naskrecki_NewYork})
	and a bound $N=N(g(C),p,\chi,s)$ for `wild' ordinary elliptic curves where $\chi$
	is the Euler characteristic of the elliptic surface attached to $E$ over $K(C)$.
\end{itemize}

\subsection{Our results}

All previous Zsigmondy bound estimates exclude the case of supersingular curves.
In this paper, we consider the case of function fields $F=K(C)$
and assume $j(E)\in K$,
which includes the case of supersingular~$E$.
In a companion paper we will deal with the case $j(E)\in F\setminus K$,
where we extend
the results of Naskręcki~\cite{Naskrecki_NewYork}
to arbitrary characteristic and improve the bound~$N$.

In the ordinary case, we prove a bound $N=2$
and show that it is optimal.
In the supersingular case in characteristic $p$, we show that the terms
$D_n$ for $n>8p$ have a primitive divisor
if and only if $p\nmid n$,
and we give a sharp version for every characteristic.

In more detail, the main results are as follows. 

\newcommand{\thmconstEcontents}{
	Let $F$ be the function field of a smooth, projective, geometrically irreducible curve over a field $K$
	of characteristic $p\geq 0$.
	Let $E$ be an elliptic curve over $K$ and $P\in E(F)\setminus E(K)$ a point of infinite order.
	Let $n$ be a positive integer,
	\begin{enumerate}
		\item
		if $p\nmid n$ or $E$ is ordinary, then $D_n$ has a primitive valuation,
		\item
		if $p\mid n$ and $E$ is supersingular, then $D_n = p^2D_{n/p}$ has no primitive valuation.
	\end{enumerate}
}
\newcommand{\thmconstjcontents}{
	Let $F$ be the function field of a smooth, projective, geometrically irreducible curve
	over a field~$K$.
	
	Let $E$ be an ordinary elliptic curve over $F$ and let $P\in E(F)$
	be a point of infinite order
	such that $j(E)\in K$, but the pair $(E,P)$ is not constant, cf.\  Definition~\ref{def:const}.
	Then for all integers $n>2$, the term
$D_n$ has a primitive valuation.

Conversely, 
for all ordinary $j$-invariants $j\in K$
there exist an elliptic curve $E/F$ with $j(E)=j$
and a point $P\in E(F)$ of infinite order such that
the terms $D_1$ and $D_2$ do not have a primitive valuation
and there exist an elliptic curve $E/F$ with $j(E)=j$
and a point $P\in E(F)$ of infinite order 
such that all terms $D_n$ for $n\geq 1$
have a primitive valuation.
}

\theoremstyle{plain}
\newtheorem*{theoremconstj}{Theorem A}
\newcommand{\theoremconstjintro}{
	\begin{theoremconstj}[Theorem \ref{thm:constj}]
		\thmconstjcontents
\end{theoremconstj}}

\newtheorem*{theoremsupersingularsharp}{Theorem B}
\newcommand{\theoremsupersingularsharpintro}{
		\begin{theoremsupersingularsharp}[Theorem \ref{thm:supersingularsharp}]
			\theoremsupersingularsharpcontents{\ref{tab:table1}}
	\end{theoremsupersingularsharp}	}

\newtheorem*{theoremconstE}{Theorem C}
\newcommand{\theoremconstEintro}{
	\begin{theoremconstE}[Theorem \ref{thm:constE}]
		\thmconstEcontents
\end{theoremconstE}}

\theoremconstjintro

\newcommand{\yesmark}{yes}
\newcommand{\nomark}{no}
\newcommand{\bothmark}{*}
\newcommand{\yesentry}[1]{\yesmark}
\newcommand{\noentry}[1]{\nomark}
\newcommand{\bothentry}[2]{\bothmark}

\newcommand{\theoremsupersingularsharpcontents}[1]{
	Let $F$ be the function field of a smooth, projective, geometrically irreducible curve over a field
	of characteristic $p>0$.
	Let $n$ be a positive integer.
	
	If the entry corresponding to $n$ and $p$ in Table~#1 is
	`$\yesmark$' (respectively `$\nomark$'), then for every supersingular
	elliptic curve $E$ over~$F$, and every $P\in E(F)$ with 
	$(E,P)$ non-constant and $P$ of infinite order, the term $D_{n}$ has
	a (respectively no) primitive valuation.
	
	If the entry is `$\bothmark$', then there exist $E$ and $P$ as in the previous
	paragraph such that $D_{n}$ has a primitive valuation and there exist $E$ and $P$
	such that $D_{n}$ has no primitive valuation.}

\theoremsupersingularsharpintro

\newcommand{\maintablecontents}{
	\begin{align*}
	&p=2\\ &
 \begin{array}{|c|c|c|c|c|c|c|c|c|c|c|}
	\hline
	n & 1 & 2 (=p) & 3 & 4 (=2p) & 6 (=3p) & 8 (=4p) & \substack{\text{odd}\\n>4} & \substack{\text{even}\\n>8}\\
	\hline
	D_{n} & \bothentry{B}{\ref{it:char2j08}} &
	 \bothentry{\ref{it:char2_j0_2and4}}{B\ref{it:char2j08}} &
	  \bothentry{B}{\ref{it:char2j08}} &
	   \bothentry{\ref{it:char2_j0_2and4}}{B\ref{it:char2j08}} &
	    \bothentry{\ref{it:char2j08}}{B} &
	     \bothentry{\ref{it:char2_j0_2and4}\ref{it:char2j08}}{B}
	& \yesentry{A} & \noentry{A}\\
	\hline
    \end{array}	\\
	&p=3\\\ & \begin{array}{|c|c|c|c|c|c|c|c|c|c|c|}
\hline
n & 1 & 2 & 3 (=p) & 6 (=2p)& 9 (=3p) & \substack{n>3\\ 3\nmid n} & \substack{n>9\\ 3\mid n}\\
\hline
D_{n} & \bothentry{B}{\ref{it:ordsharp1}\ref{it:char33}} &
 \bothentry{B}{\ref{it:ordsharp1}\ref{it:char33}} &
 \bothentry{\ref{it:ordsharp1}}{B\ref{it:char33}} &
 \bothentry{\ref{it:ordsharp1}\ref{it:char33}}{B} &
 \bothentry{\ref{it:char33}}{B} & 
 \yesentry{A} &
 \noentry{A}\\
\hline
\end{array}	\\
	&p\equiv 1\ \mathrm{mod}\ 3\\ & \begin{array}{|c|c|c|c|c|c|c|c|c|c|c|}
\hline
n & 1 & 2 & 3 & p & 2p & 3p & \substack{n>3\\ p\nmid n} & \substack{n>3p\\ p\mid n} \\
\hline
D_{n} & \bothentry{B}{\ref{it:ordsharp1}} & 
\bothentry{B}{\ref{it:ordsharp1}} & 
\yesentry{CD} &
\bothentry{\ref{it:ordsharp1}}{B} &
 \bothentry{\ref{it:ordsharp1}}{B} &
  \noentry{D} &
   \yesentry{A} &
    \noentry{A}\\
\hline
\end{array}	\\
	&p\equiv 2\ \mathrm{mod}\ 3,\quad p\not=2\\
	 & \begin{array}{|c|c|c|c|c|c|c|c|c|c|c|}
\hline
n & 1 & 2 & 3 & p & 2p & 3p & \substack{n>3\\ p\nmid n} & \substack{n>3p\\ p\mid n}\\
\hline
D_{n} & \bothentry{B}{\ref{it:ordsharp1}} &
 \bothentry{B}{\ref{it:ordsharp1}} &
  \yesentry{C} &
   \bothentry{\ref{it:ordsharp1}}{B} &
    \bothentry{\ref{it:ordsharp1}}{B} &
     \bothentry{\ref{it:D3pP}}{B} &
      \yesentry{A} &
       \noentry{A}\\
\hline
\end{array}	
    \end{align*}
}
\begin{table}[htb]
	\maintablecontents
    \label{tab:table1}
    \caption{Table mentioned in Theorem~\ref{thm:supersingularsharp}}
\end{table}

In the case where $E$ itself is defined over $K$
(and not just its $j$-invariant),
the result is much stronger, as follows.

\theoremconstEintro

\subsection{Alternative definition}
\label{sec:alternativedef}

We now give a more standard, but more technical, definition
of elliptic divisibility sequences over function fields
in the case of perfect base fields~$K$.
It is proven in \cite[Lemma~5.2]{IMSSS}
that this defines the same sequence $(\dn)_n$
in the case of number fields~$K$;
and the proof at loc.~cit.~extends to perfect fields~$K$.

Let $E$ be an elliptic curve over the function field $K(C)$ of
a smooth, projective, geometrically irreducible curve
$C$ over a perfect field~$K$.
Let $S$ be the Kodaira--N\'{e}ron model of $E$,
i.e., a smooth, projective surface with a relatively
minimal elliptic fibration $\pi:S\rightarrow C$ with
generic fibre $E$ and a section $O:C\rightarrow S$,
cf.\ \cite[\S 1]{Shioda_Mordell_Weil},
\cite[III, \S 3]{Silverman_book2}.
For example, if the curve $E$ is constant (that is, defined over~$K$),
then we can take $S=E\times C$
with the natural projection $\pi:E\times C\rightarrow C$.

	Let $P$ be a point of infinite order in the Mordell--Weil group $E(K(C))$. We define a family of effective divisors $D_{nP}\in\Div(C)$ parametrised by natural numbers $n$.
	For each $n\in\mathbf{N}$ the divisor $D_{nP}$ is the pull-back of
	the image $\overline{O}$ of the section $O$ through the morphism
	$\sigma_{nP}:C\rightarrow S$ induced by the point~$nP$, that is,
	\[D_{nP}=\sigma_{nP}^{*}(\overline{O}).\]

The delicate issues with non-perfect coefficient fields~$K$
are discussed in detail in Section~\ref{sec:additional_examples}
and Example~\ref{ex:inseparable_example}.

\subsection{Known results about divisibility sequences over function fields}

Elliptic divisibility sequences over function fields
$F=K(C)$ and related sequences were discussed in several places.
We collect some known results here.

First of all, they satisfy the \emph{strong divisibility property}
\begin{equation}\label{eq:strongdivproperty}
\gcd(B_m,B_n) = B_{\gcd(m,n)}
\end{equation}
for all positive integers $m, n$, where
$\gcd(B_m,B_n):=\sum_{v} \min\{v(B_m), v(B_n)\}[v]$.
Indeed, 
the proof in e.g.~\cite[Lemma~3.3]{Streng_Divisibility_CM} carries over.

Theorem~1.5 of \cite{IMSSS}
shows that in case $E/K$ with $K$ a number field
(and again $P\in E(F)\setminus E(K)$)
the set of prime numbers $n$ such that $\dn-\dm{1}$
is irreducible has positive lower Dirichlet density.

Cornelissen and Reynolds~\cite{Cornelissen_Reynolds}
study perfect power terms
in the case $j(E)\in F\setminus F^p$ for global function fields~$F$ of characteristic $p\geq 5$.
Everest, Ingram, Mah\'e, and Stevens study primality of terms
of elliptic divisibility sequences for $K(C)=\QQ(t)$
in the context of magnified sequences,
see \cite[Theorem~1.5]{Everest_Ingram_Mahe_Stevens_Uniform_EDS}.
Silverman~\cite{Silverman_Common_divisors} and Ghioca-Hsia-Tucker~\cite{Ghioca_Hsia_Tucker} study the
common subdivisor for two simultaneous divisibility
sequences on elliptic curves over $K(t)$, where $K$
is a field of characteristic~0. 

In a broad context, Flatters and Ward~\cite{Flatters_Ward} prove an
analogue of Theorem~\ref{thm:constj} for divisibility sequences of Lucas type for polynomials and
Akbar-Yazdani~\cite{Akbary_Yazdani} study the greatest degree of
the prime factors of certain Lucas polynomial divisibility sequences.

Hone and Swart~\cite{Hone_Swart} study examples of Somos 4
sequences over $K(t)$, which are constructed from specific elliptic divisibility sequences.
They construct a certain elliptic surface and show that
the corresponding sequence is a sequence of polynomials.

\subsection{Overview and main ideas of the proof}

The main idea behind the proof is to reduce to the case
where $E$ is defined over the base field $K$ of $F=K(C)$.
In that case $P\in E(F)$ can
be viewed as a dominant morphism $C\rightarrow E$ over~$K$.
The primitive valuations of $D_n$ then are exactly
the pull-backs of points of order $n$ on~$E$, 
which gives Theorem~C.
For details, see Section~\ref{sec:constant_curves}.

For elliptic curves over $F$ where only the $j$-invariant
is in~$K$, we find an elliptic curve $\cstE$ over $K$
with the same $j$-invariant and an isomorphism
$\phi:E\rightarrow \cstE$ over $\overline{F}$.
Then Theorem~\ref{thm:constE} applies to the sequence
$(\dnpr)_n$ obtained from $(\cstE, \phi(P))$.
See Section~\ref{sec:relating_constant_E_to_constant_j_globally}.

At that point, we know exactly which terms of $(\dnpr)_n$
have primitive valuations, and the goal
is to conclude which terms of $(\dn)_n$ have primitive valuations.

For this, we look at the
\emph{rank of apparition} $m(v)$ of a valuation $v$ of $F$
in the sequence $(D_{nP})_n$, which
is the positive integer
\[m(v)=m(P, v) :=\min\{n\geq 1: \ord_{v}(D_{nP})\geq 1 \},\]
or $\infty$ if the set is empty.
A valuation $v$ is primitive in the term $D_{nP}$ if and only if $n=m(v)$. 

The key to our proof is to see how much the rank
of apparition $m(v)$ of a valuation $v$ of $\overline{F}$ can vary
between the sequences $(\dn)_n$ and $(\dnpr)_n$.
Section~\ref{sec:relating_constant_E_to_constant_j_locally} shows that
this does not vary much, and bounds the variation
in terms of the component group of the special fibre of the N\'eron model.

This is already enough to get a weaker version
of the main results, which is not sharp, but is already
uniform (Theorems \ref{thm:constjweak} and~\ref{thm:supersingular}).

In Section \ref{sec:component_groups} we prove two auxiliary
results about the order of a point $P$ in the component group at~$v$.
This is needed in the proof of the main theorems to obtain a sharp result. 

In Section \ref{sec:the_third_term_when_j_0} we show that the
term $D_{3P}$ for sequences in characteristic $\neq 2,3$ always
has a primitive valuation if $j(E)=0$. 
This is also needed in order to obtain a sharp result.

Section \ref{sec:additional_examples} contains examples
which we use to show that our main theorems are optimal,
that is, to prove the converse statement in
Theorem~A
and the $\bothmark$-entries
in Theorem~B.

Finally, in Section \ref{sec:proof_of_the_main_theorem} we
combine all of the above into a proof of Theorems A and~B.

\subsection*{Acknowledgements}

The authors would like to thank Peter Bruin and Hendrik Lenstra
for helpful discussions and the
anonymous referee for comments that improved the exposition.

\section{Constant curves}\label{sec:constant_curves}

Let $C$ be a smooth, projective, geometrically irreducible curve
over a field $K$ and let $F = K(C)$ be its function field.
Let $E/F$ be an elliptic curve and $P\in E(F)$ a point.
For a field extension $M\supset L$ and an elliptic curve $E'$ over~$L$,
let $E_M'$ be the base change of $E'$ to~$M$.
\begin{definition}\label{def:const}
	We say that $E$ is \emph{constant} if there exists an elliptic curve $\cstE/K$ and an isomorphism $\phi:E\rightarrow \cstEF$ defined over $F$.
	
	We say that the pair $(E,P)$ is \emph{constant} if there exist such $\cstE$
	and $\phi$ that also satisfy $\phi(P)\in \cstE(K)$.
	
	We say that the $j$-invariant $j(E)$ of the curve $E/K(C)$ is \emph{constant} if $j(E)\in K$.
\end{definition}

\begin{lemma}\label{lem:constconst}
	The pair $(E,P)$ is \emph{constant} if and only if $E$ is constant and
	for \emph{all} elliptic curves $\cstE/K$
	and isomorphisms $\phi : E\rightarrow \cstEF$ we have $\phi(P)\in \cstE(K)$.
\end{lemma}
\begin{proof}
	The `if' implication follows from the definition, so it is enough to prove the `only if' implication.
	Suppose that $(E,P)$ is constant.
	There exists an elliptic curve $\widetilde{E}$ defined over $K$ and
	an isomorphism $\phi:E\rightarrow \widetilde{E}_{F}$ of $F$-curves
	such that $\phi(P)\in\widetilde{E}(K)$.
	Let $\cstE'$ be an elliptic curve over $K$ and
	$\phi':E\rightarrow \cstE'_F$ another $F$-isomorphism.
	Let $\phi\circ\phi'^{-1}:\cstE'_{F}\rightarrow \cstE_{F}$ denote the corresponding isomorphism of $F$-curves. 
	It follows that the curves $\cstE'_{F}$ and $\cstE_{F}$ have equal $j$-invariant and since $\cstE$ and $\cstE'$ are defined over $K$ there exists a $\overline{K}$-isomorphism $\psi:\cstE_{\overline{K}}\rightarrow \cstE'_{\overline{K}}$. 
	Let $F'$ denote the function field of the curve $C_{\overline{K}}$. We have
	\[
	\eta\circ\psi^{\phantom{\prime}}_{F'}\circ \phi^{\phantom{\prime}}_{F'}
	=\phi'_{F'}
	\]
	for some $\eta\in \Aut(\cstE'_{F'})=\Aut(\cstE'_{\overline{K}})$.
	Since $P\in E(F)$, we have $\phi'(P)\in \cstE'(F)$.
	From our assumption it follows that $\phi(P)\in \cstE(K)$ and
	hence $\phi'(P) =
	(\eta\circ\psi_{F'}\circ \phi_{F'})(P)\in \cstE'(\overline{K})$.
	Combining these statements, we get $\phi(P)\in \cstE'(\overline{K}\cap F)$.
	As $C$ is smooth and geometrically irreducible, it is geometrically
	integral, hence by \cite[Corollary 3.2.14(c)]{Liu},
	we get $\overline{K}\cap F = K$. 
\end{proof}

\subsection{\texorpdfstring{Constant $E$}{Constant E}}

Suppose that $E$ is constant. Then without loss of generality we consider $E=\cstEF$.
Then $P\in E(F)$ can be interpreted as a morphism of curves $P : C\rightarrow \cstE$ defined over~$K$
as follows. We give two interpretations, both leading to the same morphism.

Consider the constant elliptic surface $(S,\pi, C)$ where $S = \cstE\times C$ and $\pi: S\rightarrow C$ is the projection on the second factor.
Every point $P$ on the generic fibre $E$ corresponds to a unique section
$\sigma_{P}:C\rightarrow S$.
Composition $\mu\circ \sigma_{P}:C\rightarrow \cstE$ of $\sigma_{P}$ with
the projection $\mu:S\rightarrow \cstE$ on the first factor
is a morphism defined over~$K$.
By abuse of notation we denote the morphism $\mu\circ\sigma_{P}$ by~$P$.

Equivalently, the point $P\in E(F)$ has coordinates in $F=K(C)$,
hence defines a rational map from $C$ to~$E$.
All such rational maps are morphisms as $C$ is smooth and $E$ is projective.

Applying this abuse of notation to $nP$ too, we get  
$nP = [n]\circ P$.

Note that the pair $(E,P)$ is constant if and only if the morphism $P : C\rightarrow \cstE$
is a constant morphism, or equivalently, maps to a single point.

\subsubsection{Constant $P$}\label{sec:constP}
If $P$ maps to a single point, then so does $[n]\circ P$. In particular, for all $n$
either $nP=O$ or the images of $[n]\circ P$ and $O$ are disjoint.
If $P$ has infinite order, then this gives for all $n$:
\begin{equation}\label{eq:constzero}
\dn=0.
\end{equation}

\subsubsection{Non-constant $P$}
Let us assume in this section that $E$ is constant and
the morphism $P:C\rightarrow\cstE$ is non-constant.
\begin{theorem}\label{thm:constE}
	\thmconstEcontents
	\end{theorem}
\begin{proof}
	We have $\dn = ([n]\circ P)^* (O) = P^* [n]^* (O)$.
	Note that 
	\begin{equation}\label{eq:pullback_of_identity}
	[n]^*(O) = \deg_i([n]) \sum_{Q\in E(\overline{K})[n]} (Q)
	\end{equation}
	where $\deg_i([n])$ denotes the inseparable degree of~$[n]$.
	If $p\mid n$ and $E$ is supersingular,
	then the endomorphism $[p]$ is purely inseparable of degree $p^2$,
	hence $\deg_i([p\cdot\frac{n}{p}])=p^2\deg_i([\frac{p}{n}])$ and $E(\overline{K})[p]=\{O\}$;
	so formula \eqref{eq:pullback_of_identity} gives $D_n = p^2 D_{n/p}$, which proves~(2).
	Moreover, if $E$ is ordinary or $p\nmid n$, then $E(\overline{K})$ contains a point $Q_n$ of order~$n$.
	Since $P$ is a dominant morphism,
	there is a point that maps to $Q_n$ under~$P$,
	and the valuation associated with such a point is a primitive valuation of $\dn$.
\end{proof}
\begin{remark}
	The existence of a cover $P : C \rightarrow \cstE$
	implies that $C$ has genus greater than or equal to~$1$.
	
	In fact, as all such covers factor through the identity map
	$\mathrm{id}:\widetilde{E}\rightarrow \widetilde{E}$, we see
	that for every elliptic curve
	$\cstE/K$, there is one prototypical example given by
	$P = \mathrm{id}:\widetilde{E}\rightarrow \widetilde{E}$.
	In other words, this example has
	$C=\cstE$ and $S=\cstE\times \cstE$.
	The point $P\in E(K(E))$ corresponds to the morphism
	$\mu\circ\Delta$ where $\Delta:\cstE\rightarrow \cstE\times\cstE$
	is the diagonal map and $\mu$ is the projection on the first factor.
\end{remark}
\begin{example}\label{ex:ssconst}
	Let $C = \cstE : y^2 = x^3 + x$ over $K = \FF_3$
	and $P = (x,y) \in \cstE(K(\cstE))$.
	Let $i$ be a square root of $-1$ in a quadratic
	extension of $\FF_3$.
	Then
	\begin{align*}
	\cstE(\overline{K})[1] &= \{O\}\\
	\cstE(\overline{K})[2] &= \cstE(\overline{K})[1]\cup \{(0,0), (\pm i, 0)\}\\
	\cstE(\overline{K})[3] &= \cstE(\overline{K})[1]\\
	\cstE(\overline{K})[4] &= \cstE(\overline{K})[2]\cup \{(1,\pm i), (-1, \pm 1), (\pm i + 1, \pm i), (\pm i - 1, \pm 1)\},
	\end{align*}
	where all the signs are independent.
	In particular, we obtain
	\begin{align*}
	\dm{1} &= O\\
	\dm{2} &= \dm{1} + (0,0) + (i,0)+(-i,0)\\
	\dm{3} &= 9O\\
	\dm{4} &= \dm{2} + \sum_{s\in\{\pm 1\}}
	\left[
	\begin{array}{l}
	(1,si) + (-1,s) + (i+1,si) \\+ (-i+1, si) + (i-1, s) + (-i-1, s)
	\end{array}
	\right]\\
	\dm{6} &= 9\dm{2}.
	\end{align*}
	By symmetry, for $b\not=0$ the points $(a,b)$ and $(a,-b)$ only appear together.
	Because of that, we introduce the following notation.
    Let $$D_{m}'=\left\{\begin{array}{ll}
     \dm{m}-\ord_O(\dm{m}) O & \qquad \textrm{if }2\nmid m,\\
     \dm{m}-\ord_O(\dm{m})\dm{2} & \qquad \textrm{if }2\mid m.
    \end{array}   \right.$$
	The divisor $D_m'$ is the pull-back $x^{*}\delta_m$
	of an effective divisor $\delta_m$ on the affine $x$-line~$\AA^1$.
	Let $p(m)$ denote a monic polynomial with divisor of zeroes
	equal to~$\delta_{m}$.
	We get that the divisor $\dn$ has the form
	\begin{equation}
	\dn = a(n) (O) + b(n) \div_0(y) + \div_0(p(n)),
	\end{equation}
	where
	\begin{align*}
	a(n) &= 9^{\mathrm{ord}_3(n)},\\
	b(n) &= \left\{\begin{array}{ll}0 &\mbox{if $n$ is odd,}\\
	a(n) &\mbox{if $n$ is even,}\end{array}\right.
	\end{align*}
	and we present below only the factorisation of the polynomial $p(n)$.
	{\small
	\begin{align*}
	p(1) &= 1 \\
	p(2) &= 1 \\
	p(3) &= 1 \\
	p(4) &= (x + 1) \cdot (x + 2) \cdot (x^{2} + x + 2) \cdot (x^{2} + 2 x + 2) \\
	p(5) &= (x^{4} + x + 2) \cdot (x^{4} + 2 x + 2) \cdot (x^{4} + x^{2} + 2) \\
	p(6) &= 1
	\\
	p(7) &= (x^{3} + x^{2} + 2 x + 1) \cdot (x^{3} + 2 x^{2} + 2 x + 2) \cdot (x^{6} + 2 x^{4} + x^{2} + 1) \\
	&\quad \cdot (x^{6} + x^{5} + 2 x^{4} + 2 x^{3} + 2 x^{2} + 2 x + 1) \cdot (x^{6} + 2 x^{5} + 2 x^{4} + x^{3} + 2 x^{2} + x + 1) \\
	p(8) &= (x + 1) \cdot (x + 2) \cdot (x^{2} + x + 2) \cdot (x^{2} + 2 x + 2) \cdot (x^{4} + x^{3} + x^{2} + x + 1) \\
	&\quad \cdot (x^{4} + x^{3} + x^{2} + 2 x + 2) \cdot (x^{4} + x^{3} + 2 x^{2} + 2 x + 2)\\
	&\quad \cdot (x^{4} + 2 x^{3} + x^{2} + x + 2) \cdot (x^{4} + 2 x^{3} + x^{2} + 2 x + 1)\\
	&\quad \cdot (x^{4} + 2 x^{3} + 2 x^{2} + x + 2) \\
	p(9) &= 1 \\
	p(10) &= (x^{4} + x + 2) \cdot (x^{4} + 2 x + 2) \cdot (x^{4} + x^{2} + 2) \cdot (x^{4} + x^{2} + x + 1) \\
	&\quad\cdot (x^{4} + x^{2} + 2 x + 1) \cdot (x^{4} + 2 x^{2} + 2) \cdot (x^{4} + x^{3} + 2) \cdot (x^{4} + x^{3} + 2 x + 1)\\
	&\quad \cdot (x^{4} + x^{3} + x^{2} + 1) \cdot (x^{4} + 2 x^{3} + 2) \cdot (x^{4} + 2 x^{3} + x + 1) \cdot (x^{4} + 2 x^{3} + x^{2} + 1) \\
	p(11) &= (x^{10} + x^{7} + x^{5} + x^{4} + 2 x^{3} + 2 x^{2} + x + 2)\\
	&\quad \cdot (x^{10} + 2 x^{7} + 2 x^{5} + x^{4} + x^{3} + 2 x^{2} + 2 x + 2) \\
	&\quad\cdot (x^{10} + 2 x^{8} + x^{6} + x^{5} + x^{2} + 2 x + 2) \cdot (x^{10} + 2 x^{8} + x^{6} + 2 x^{5} + x^{2} + x + 2) \\
	&\quad\cdot (x^{10} + 2 x^{8} + x^{7} + x^{6} + x^{5} + 2 x^{4} + 2 x^{3} + x^{2} + 2)\\
	&\quad \cdot (x^{10} + 2 x^{8} + 2 x^{7} + x^{6} + 2 x^{5} + 2 x^{4} + x^{3} + x^{2} + 2) \\
	p(12) &= (x + 1)^{9} \cdot (x + 2)^{9} \cdot (x^{2} + x + 2)^{9} \cdot (x^{2} + 2 x + 2)^{9}
	\end{align*}}\indent This confirms
	the equality $\dm{12} = 9\dm{4}$ and the fact that terms $\dn$ with $3\nmid n$ have primitive valuations.
\end{example}

\section{\texorpdfstring{Relating constant $E$ to constant $j$ globally}{Relating constant E to constant j globally}}\label{sec:relating_constant_E_to_constant_j_globally}

\subsection{Definitions and example}\label{sec:definitions}
Let $E$ be an elliptic curve over $F=K(C)$
and let $P\in E(F)$ be a point of infinite order.
Now suppose $j(E)\in K$.
Note that this includes the case where $E$
is supersingular by \cite[V.3.1(a)(iii)]{Silverman_arithmetic}.

The idea behind the proof of our main results is to relate the
EDS $(\dn)_n$ obtained from $P$ with constant $j$-invariant
to an EDS $(\dnpr)_n$
obtained from a point on a constant elliptic curve
and then to apply Theorem~\ref{thm:constE} to $(\dnpr)_n$.

	\begin{lemma}\label{lem:Fprime}
		Let $K$ be a field, let $C/K$ be a smooth, projective, geometrically
		irreducible curve and let $F = K(C)$.
		Let $E/F$ be an elliptic curve with $j(E)\in K$.
		
		Then there exist
		\begin{enumerate}[(i)]
			\item an elliptic curve $\cstE/K$ with $j(\cstE) = j(E)$,
			\item finite extensions $K'\supset K$ and $F'\supset F$ with $K'\subset F'$,
			\item an isomorphism $\phi : E_{F'} \rightarrow \cstE_{F'}$,
			\item a smooth, projective, geometrically irreducible curve $C'/K'$ with
			$F' = K'(C')$, and
			\item a non-constant morphism $f : C'\rightarrow C_{K'}$
			inducing the inclusion map $FK'\hookrightarrow F'$.
		\end{enumerate}	
	\end{lemma}
\begin{notation}\label{not:Fprime}
	On top of the notation of Lemma~\ref{lem:Fprime}, we use the following notation.
	Given a point $P\in E(F)$,
	\begin{itemize}
		\item[(vi)] let	$P' = \phi(P)\in\cstE(F')$,
		\item[(vii)] let $D_{nP}$ be the EDS obtained from $(E,P)$ as defined in \eqref{eq:defedseasy}, and
		\item[(viii)]
	 let $D_{nP'}$ be the EDS obtained from $(\cstE, P')$.	
	\end{itemize}
	The symbol $v$ will denote a place of $C$ and
	$v'$ will denote a place of $C'$ lying over~$v$.
\end{notation}
\begin{proof}
	(i) Let $\cstE/K$ be an elliptic curve with $j(\cstE) = j(E)$,
	let $\overline{F}$ be an algebraic closure of~$F$
	and let $\overline{K}\subset \overline{F}$ be an algebraic closure of~$K$.
	Then there exists an isomorphism
	$\phi_{\overline{F}}: E_{\overline{F}}\rightarrow \cstE_{\overline{F}}$
	by \cite[Proposition III.1.4(b)]{Silverman_arithmetic}.
	Let $F''\subset\overline{F}$ be generated over $F$ by
	the coefficients $b_1$, $b_2,\ldots, b_r\in \overline{F}$ of~$\phi_{\overline{F}}$.
	
	If $K$ is perfect, then we take $F' = F''$ and $K' = (\overline{K}\cap F'')$
	(which satisfy (ii) and (iii))
	and find by \cite[Remark II.2.5]{Silverman_arithmetic}
	a curve~$C'$ satisfying (iv). The inclusion $FK'\subset F'$ then
	gives the morphism of~(v).
	
	If $K$ is not perfect, then this
	construction does not always give a smooth curve
	(see Example~\ref{ex:nonsmooth}),
	so we do some additional steps in our construction.
	
    The field $F''' := F''\overline{K}$ is a
	finitely generated extension of transcendence degree~$1$
	over the algebraically closed field~$\overline{K}$,
	hence is the function field
	of a smooth, projective, (geometrically) irreducible curve $C'_{\overline{K}} / \overline{K}$
	by \cite[Remark II.2.5]{Silverman_arithmetic}.
	
	The inclusion $F\overline{K}\subset F'''$ induces
	a non-constant rational map $f_{\overline{K}} : C'_{\overline{K}} \rightarrow C_{\overline{K}}$,
	which is a morphism as the curves are regular and projective.
	
	Choose embeddings $i:C\rightarrow \mathbf{P}^m$ over $K$
	and $j : C'_{\overline{K}}\rightarrow \mathbf{P}^n$ over $\overline{K}$.
	Then embed $C'_{\overline{K}}$ into $\mathbf{P}^n\times (\mathbf{P}^1)^r$
	via $j \times b_1\times\cdots \times b_r$.

	(ii,iv) Then let $K'$ be generated over $K$ by the coefficients
	of a system of defining equations of
	$C'_{\overline{K}}\subset \mathbf{P}^n\times (\mathbf{P}^1)^r$
	and~$f_{\overline{K}}$.
	From now on we view $C'_{\overline{K}}$ as a projective curve $C'$ over~$K'$,
	which is smooth and irreducible over~$\overline{K}$.
	In particular, the curve $C'$ is smooth, projective and
	geometrically irreducible over~$K'$.
	
	Moreover, the field $F' := K'(C')$
	contains the coefficients $b_1,\ldots, b_r$
	of $\phi_{\overline{F}}$ since they are coordinate functions on the $r$ copies of
	$\mathbf{P}^1$.
	(iii) In particular the morphism 
	$\phi_{\overline{F}}$ can be viewed as a morphism $\phi:E_{F'}\rightarrow \cstE_{F'}$.
	(v) Similarly $K'$ contains the coefficients of $f_{\overline{K}}$,
	so $f_{\overline{K}}$ can be viewed as a morphism $f:C'\rightarrow C_{K'}$.
\end{proof}
The following two examples illustrate why the proof of Lemma~\ref{lem:Fprime}
is so complicated for non-perfect base fields~$K$.
The reader who is interested mostly in the perfect case may wish to
skip ahead to Example~\ref{ex:ss}.
\begin{example}\label{ex:nonsmooth}
	Here is an example to show that we cannot just take $F' = F''$
	if $K$ is non-perfect.
	Let $K = \FF_2(b)$, $F = K(u)$, and
	$E : y^2 + u^3y = x^3 + b$, so $j(E) = 0$.
	Take $\cstE : y^2 + y = x^3$
	and $\phi : (x,y) \mapsto (u^{-2} x, u^{-3} (y + t))$ where
	$t\in\overline{F}$ satisfies
	\begin{equation}\label{eq:nonsmooth}t^2 + u^3t + b = 0.
	\end{equation}
	Then $F'' = F(t) = K(u,t)$
	is the function field of a regular, projective, geometrically integral curve over~$K$
	with affine open part given by \eqref{eq:nonsmooth},
	but this curve is not smooth.
	Indeed the given model is smooth exactly outside $u=0$ and is regular at the place
	$u=0$.
\end{example}
\begin{example}\label{ex:extracoordinates}
	To motivate why we embed $\smash{C'_{\overline{K}}}$
	in such a complicated way in the proof,
	let $K$ be a field of characteristic $\neq 3$
	in which $-1$ is not a square.
		Let $F = K(t)$,
		$E : y^2 = x^3 + t^2$,
		$\cstE : y^2 = x^3 - 1$,
		$s = \sqrt[3]{t}\in\overline{F}$, and $\sqrt{-1}\in \overline{F}$.
		Take $\phi : (x,y)\mapsto (-s^{-2}x, \sqrt{-1}s^{-3}y)$,
		so $F'' = F(\sqrt{-1}s)$.
		Then $F''' = \overline{K}(s)$, so we can take $C' = \mathbf{P}^1$
		with $s$ as coordinate.
		The map $f$ is then given by $t = s^3$.
		
		Now we can take $i$ and $j$ to be the identity map.
		If we had defined $K'$ using only $f$ and the images of $i$
		and $j$, then we would have gotten $K' = K$ and $F' = K'(s)=K(s)$.
		But then $\phi$ is not defined over~$F'$.
\end{example}

Here is an example where we compute both $(\dn)_n$ and $(\dnpr)_n$
and compare them.

\begin{example}\label{ex:ss}
	Let $E$ be the supersingular elliptic curve over $F = \FF_3(t)=\FF_3(\P^1)$
	given by
	\begin{equation}
	E : y^2 = x^3 + tx -t
	\end{equation}
	and let $P = (1,1)\in E(F)$.
	We start by computing $\dn$ for a few values of~$n$.
	The discriminant $\Delta(E)$ is $-t^3$, hence the given model is minimal for all finite places
	of $\P^1$.
	Therefore, we can compute these valuations of $\dn$
	by computing the square root of the denominator of~$x(nP)$.
	For the valuation at infinity, we take $(x',y') = (t^{-2}x,t^{-3}y)$,
	so 
	\[ E : {y'}^2 = {x'}^3 + t^{-3}{x'} - t^{-5},\qquad
	(x'(P), y'(P)) = (t^{-2},t^{-3}),\]
	which is minimal since the discriminant $-t^{-9}$ has valuation~$9$.
	
	To keep the notation short, we write $p(t)\cdot \infty^k$
	with $p(t)\in \FF_3(t)$ to mean
	$\div_0(p(t)) + k(\infty) := \sum_{v\not=\infty}\ord_v(p(t))(v) + k(\infty)$.
	In this notation, we compute
    \begin{align*}
    \dm{1} &= \dm{2}=1, &
    \dm{3} &= t^{2},  \\
    \dm{4} &= (t + 2)^{3},  &
    \dm{5} &= (t^{2} + t + 2)^{3}, \\
    \dm{6} &= t^{2} \cdot \infty^{ 3 }, &
    \dm{7} &= (t + 1)^{3} \cdot (t^{3} + t^{2} + t + 2)^{3}, \\
    \dm{8} &= (t + 2)^{3} \cdot (t^{4} + 2 t^{3} + t + 1)^{3}, &
    \dm{9} &= t^{20}, \\
    \dm{10} &= (t^{2} + t + 2)^{3} \cdot (t^{2} + 2 t + 2)^{3} \cdot (t^{4} + t^{3} + t^{2} + 1)^{3}, \\
    \dm{11} &= (t^{10} + 2 t^{9} + 2 t^{8} + t^{7} + t^{6} + 2 t^{5} + 2 t^{4} + 2 t^{3} + 2 t^{2} + 1)^{3}, &
    \dm{12} &= t^{2} \cdot (t + 2)^{27} \cdot \infty^{ 3 }.
    \end{align*}
	All terms $\dn$ with $3\nmid n$ listed here have a
	primitive valuation except $\dm{1}$ and $\dm{2}$.
	All terms $\dn$ with $3\mid n$ listed here have
	no primitive valuation except $\dm{3}$ and $\dm{6}$.
	
	As $j(E)=0$, we find that $E$ is isomorphic over $\overline{F}$ to
	\begin{equation}
	\cstE : Y^2 = X^3 + X.
	\end{equation}
	
	Next, we look for an isomorphism $\phi : E_{\overline{F}}\rightarrow \cstE_{\overline{F}}$.
	All isomorphisms are given in case II of
	the proof of Proposition A.1.2(b) of
	Silverman~\cite{Silverman_arithmetic} as
	\begin{align}\label{eq:isomj0p3}
	X &= u^2 x + r,\qquad & Y &= u^3 y,\qquad\mbox{where}\\
	u^4 &= 1/t,\qquad & 0 &= 
	r^3 + r + u^2.\label{eq:uandv}
	\end{align}
	We use the notation $v = -r$ and solve for $u$ and~$v$
	in~\eqref{eq:uandv}.
	Choose a $4$th root $u\in \overline{F}$ of $1/t$,
	and 
	take $v\in\overline{F}$ such that $v^3+v=u^2$.
	Then $F' = F(u,v)$ is an extension of $F$ of degree $12$
	and is the function field of the curve
	\begin{equation}
	C' : u^2 = v^3+v
	\end{equation}
	over $K'=K$.
	The inclusion $F\rightarrow F'$ corresponds to the
	projection
	\begin{equation*}
	u^{-4} : C' \rightarrow \P^1 : (v,u)\mapsto u^{-4},
	\end{equation*}
	which is a $12$-fold covering.
	The isomorphism $\phi$ given by \eqref{eq:isomj0p3} is
	\begin{align}
	\phi : E_{F'} & \rightarrow \cstEFpr\\
	(x,y) & \mapsto (u^2x -v, u^3 y).\nonumber
	\end{align}
	Then \begin{equation}P' = \phi(P) = (u^2-v, u^3) = (v^3,u^3)=\mathrm{Frob}_3((v,u))\in \cstE(F').\end{equation}
	In other words, if we identify $C'$ with $\cstE$ via $(X,Y) = (v,u)$, then
	$P': C'\rightarrow \widetilde{E}$ is
	the (purely inseparable) $3$rd power Frobenius endomorphism
	$\mathrm{Frob}_3$
	(of degree~$3$).
	In particular, the EDS $(\dnpr)_n$ obtained from $(\cstE,P')$
	is $3$ times the EDS of Example~\ref{ex:ssconst}.
\end{example}

\subsection{\texorpdfstring{The point $P'$ is non-constant}{The point Pprime is non-constant}}
Next we show that if $(E,P)$ is non-constant
(cf.~Definition~\ref{def:const}),
then the point $P'=\phi(P)\in \cstE(F')$ of Notation~\ref{not:Fprime}
is non-constant (that is, not in $\cstE(\overline{K})$).

\begin{lemma}[Tate normal form]\label{lem:tatenormalform}
	Let $E$ be an elliptic curve over a field $L$ and let $P\in E(L)$
	be a point
	of order $\geq 4$.
	Then there are unique $b$, $c\in L$
	and a change of coordinates over $L$ such that
	\begin{equation}
	E : y^2 + (1-c)xy - by = x^3 - bx^2,\qquad P=(0,0).
	\end{equation}
\end{lemma}
\begin{proof}
	Starting with a general Weierstrass equation, first translate
	to get $P=(0,0)$ (allowed as $P\not=O$). Then
	\[ E : y^2 + a_1xy + a_3y = x^3 + a_2x^2+a_4x.\]
	With $y \mapsto y-a_4/a_3 x$
	(allowed as $2P\not=O$), we get $a_4=0$.
	With $(x,y)\mapsto (u^2x,u^3y)$ and $u= a_2/a_3$
	(allowed as $3P\not=O$),
	we get $a_2=a_3$.
	Then let $b = -a_2$ and $c = 1-a_1$.
	This proves existence.
	
	Unicity follows as we used up all freedom for
	changes of Weierstrass equations
	$(x,y)\mapsto (u^2x + r, u^3y + u^2sx + t)$
	as in \cite[III.3.1(b)]{Silverman_arithmetic}.
\end{proof}

\begin{corollary}\label{cor:whenconst}
	Let $K$, $F$, $E$, $\cstE$, $P$, and $P'$ be as in Notation~\ref{not:Fprime}
	(this includes $j(E)\in K$).
	Suppose that $P$ has order $\geq 4$.
	If $P'$ is constant
	(that is, is in $\cstE(\overline{K})$),
	then the pair $(E,P)$ is constant
	as in Definition~\ref{def:const}.
\end{corollary}
\begin{proof}
	If $P'\in \cstE(\overline{K})$, then the Tate normal form
	of $(\cstE,P')$ has $b$, $c\in \overline{K}$.
	The Tate normal form of $(E,P)$ has $b$, $c\in F$.
	By uniqueness of the Tate normal form over $\overline{F}$,
	we get $b$, $c\in\overline{K}\cap F=K$, hence $(E,P)$
	is isomorphic over $F$ to a pair defined over~$K$.
\end{proof}

In particular, in our case where $P$ has order $\infty>4$,
if the pair $(E,P)$ is non-constant, then the point $P'$ is non-constant.

\section{\texorpdfstring{Relating constant $E$ to constant $j$ locally}{Relating constant E to constant j locally}}\label{sec:relating_constant_E_to_constant_j_locally}

\subsection{\texorpdfstring{Reduction modulo primes of curves with constant $j$}{Reduction modulo primes of curves with constant j}}
Elliptic curves with constant $j$-invariant
admit only places of good or additive reduction.
We show that the valuations $v$ of additive reduction appear early on
in the sequence $D_{nP}$ (Lemma~\ref{lem:additive}(\ref{it:additive1}--\ref{it:additive2})),
while those of good reduction appear in the same place
as in the corresponding constant sequence $\dnpr$ (Lemma~\ref{lem:constj}).

Recall that the rank of apparition $m(v)=m(P,v)$ of a valuation
$v$ of $F$ is the
smallest positive integer $n$ such that $v(\dn)>0$
(with $m(v)=\infty$ if it does not exist).

With the notation as in Notation~\ref{not:Fprime},
let $F_v$ be the completion of $F$ at~$v$.
Let $E_0(F_v)$ (respectively $E_1(F_v)$) be the subgroup of
$E(F_v)$ consisting of points that reduce to a non-singular point
(respectively the point $O$) on the reduction of the minimal
Weierstrass equation.
In particular, we have $v(D_n)>0$ if and only if $nP\in E_1(F_v)$.
Moreover the quotient $E(F_v)/E_0(F_v)$ is the component
group of the special fibre of the N\'eron model of $E$ at~$v$
(cf.\ \cite[Corollary IV.9.2]{Silverman_book2}
and \cite[Theorem 5.5]{Conrad_minimal_models}).

\begin{lemma}
	\label{lem:additive}
	Let $F$ be the function field of a smooth, projective,
	geometrically irreducible curve over a field~$K$ of characteristic $p\geq 0$.
    Let $E$ be an elliptic curve over $F$
	and let $P\in E(F)$.
	Let $v$ be a discrete valuation of~$F$
	with $v(K)=\{0\}$ and $v(F)\not=\{0\}$,
	and let $d=d_v$ be the order of $P$ in
	the component group $E(F_v)/E_0(F_v)$.
	\begin{enumerate}
		\item \label{it:constant}
		If $j(E)\in K$, then $E$ has good or additive reduction at~$v$.
		\item \label{it:additive1} If $E$ has additive reduction at $v$, 
		then
		\begin{enumerate}
			\item if $p=0$, then $m(v)= d$ or $m(v)=\infty$.
			\item if $p>0$, then $m(v)= d$ or $m(v)=dp$.
		\end{enumerate}
	\item \label{it:additive2} If $E$ has additive reduction at $v$,
	then $d\leq 4$.
		\item \label{it:const}
		If $j(E)\in K$, then
		\begin{enumerate}[(a)]
			\item\label{it:consta} if $p\not=2$, then $d\leq 3$,
			\item\label{it:constb} if $p\not=3$ and $j(E)\not=0$, then $d\mid 4$,
			\item\label{it:constc} if $p\not\in\{2,3\}$ and $j(E)\not=0$, then $d\leq 2$.
		\end{enumerate}
	\end{enumerate}
\end{lemma}
\begin{proof}
	(\ref{it:constant}) As $j(E)\in K$, we have $v(j(E))\geq 0$,
	hence $E$ does not have
	multiplicative reduction at~$v$
	by \cite[Proposition VII.5.1(b) and $j=c_4^3/\Delta$]{Silverman_arithmetic}.	
	
	(\ref{it:additive1})
Note that $m(v)$ is the order of $P$ in $E(F_v)/E_1(F_v)$.
In the additive case, the subgroup $E_0(F_v)/E_1(F_v)$ is isomorphic
to the additive group underlying the residue field of~$v$.
If $p=0$, then the latter group is torsion-free, so that $m(v)$ is $\infty$
or $d$.
If $p>0$, then the latter group has exponent~$p$, so that $m(v)$ is
$d$ or $dp$.

Write $c = \#E(F_v)/E_0(F_v)$, so $d\mid c$.
For parts (\ref{it:additive2}) and (\ref{it:const}), 
we will use tables of reduction types to find
restrictions on $c$, hence on~$d$.

If $K$ is perfect, then the reduction types
were classified by Kodaira and N\'eron and
can be found in \cite[Table 4.1 in \S IV.9]{Silverman_book2}
(equivalently \cite[Table 15.1 in Appendix~C]{Silverman_arithmetic}).
For general fields~$K$, we need the generalization
by Szydlo~\cite[Theorem~3.1 and Proposition~7.1.1]{Szydlo_reduction_types}.
All types that are in Szydlo's classification
and were not already in the Kodaira-N\'eron classification
have $c\in\{1,2\}$
by \cite[(20) on page~96]{Szydlo_reduction_types},
so we may assume that we are in one
of the cases from the Kodaira-N\'eron classification.

	(\ref{it:additive2}) In the additive reduction case, we get $\#E(F_v)/E_0(F_v)=c$
	with $c\in \{1,2,3,4\}$ by \cite[Table~4.1]{Silverman_book2}.
	
	(\ref{it:consta}) Suppose first that $K$
	is perfect.
	If $c=4$, then the reduction type is $I_n^*$,
	hence by the bottom part of \cite[Table~4.1]{Silverman_book2},
	we get $\mathrm{char}(K)=2$
	or $v(j(E))<0$.
	For general fields Theorem~5.1 and
	Tables 1 and~4 of \cite{Szydlo_reduction_types}
	give the same result.
	
    (\ref{it:constb}) In the same way,
	the case $c=3$ only happens
	when $\mathrm{char}(K)=3$ or $v(j(E))>0$.
	Indeed, the reduction type is $IV$ or $IV^*$,
	the same reference works in the perfect case,
	and in the general case one needs Table~5 in~\cite{Szydlo_reduction_types}
	instead of Table~4.
	
	Combining \eqref{it:consta} with \eqref{it:constb} gives~\eqref{it:constc}.
\end{proof}

\subsection{\texorpdfstring{Relating $(\dn)_n$ with $(\dnpr)_n$}{Relating Dn with Dnprime}}

To prove our main results, we link the EDS $(\dn)_n$ obtained
from $(E,P)$ to the EDS $(\dnpr)_n$ obtained from $(\cstE, P')$.
Let $v'$ be a valuation of $F'$ lying over a valuation $v$
of $F$.
\begin{lemma}\label{lem:constj}
	Let the notation be as in Notation~\ref{not:Fprime}
	and suppose that $P$ has infinite order.
	
	If $E$ does not have additive reduction at $v$, then
	we have 
	$v'(\dmpr{m}) = v'(\dm{m})$ for all~$m\in\ZZ_{>0}$.
\end{lemma}
\begin{proof}
	Note that $\cstEFpr$ has good reduction at all valuations of~$F'$.
	Suppose that $E$ does not have additive reduction at $v$.
	As $j(E)\in K$, we find that $E$ also has good reduction at $v$,
	hence the isomorphism $E_{F'}\rightarrow \cstEFpr$
	is an isomorphism over the local ring at~$v'$,
	which does not affect the valuation of $x(mP)$.
\end{proof}

\begin{example}\label{ex:sscontinued}
	We continue Example~\ref{ex:ss},
	so $E: y^2 = x^3 + tx - t$ and $P=(1,1)$.
	In that example, we saw that the EDS $(\dnpr)_n$
	is $3$ times the EDS of Example \ref{ex:ssconst},
	with $X=v$, $Y=u$, $u^2 = v^3 + v$ and $t = u^{-4}$.

	We compute the difference $\dnpr-\dn$
	for the first few terms.
	To help in this computation, note the following identities.
	\begin{align*}
	\div_0(t) &= \div_0(u^{-4}) = 12(O),\\
	\div_{\infty}(t) &= 4\div_0(u),\qquad\mbox{where}\qquad
	\div_0(u) = (0,0) + (i,0) + (-i,0),\\
	\intertext{and if $p$ is a polynomial
		with $p(0)\not=0$ and $p^*$ is its reciprocal, then}
	\div_0(p(t)) &= \div_0(p^*(u^4)) = \div_0(p^*((v^3+v)^2)).
	\end{align*}
	We obtain
	\begin{align*}
	\dmpr{1}-\dm{1} & & &= 3(O)\\
	\dmpr{2}-\dm{2} & & &= 3(O) + 3\div_0(u)\\
	\dmpr{3}-\dm{3} & =  27(O)-2\div_0(t) & & = 3(O)\\
	\dmpr{4}-\dm{4} & = 
	3(O) + 3\div_0(u) \\
 +\ & 3\div_0((v+1)(v+2)(v^2+v+2)(v^2+2v+2))\!\!\!\!\!\!\!\!\!\\
 +\ & -3\div_0(1+2u^4) &
	&= 3(O) + 3\div_0(u)\\
	\dmpr{5}-\dm{5} &= \\
	& \div_0((v^{4} + v + 2)  (v^{4} + 2 v + 2)  (v^{4} + v^{2} + 2))\\
	+\ &
	- \div_0(1+u^4+2u^8) & 
	&= 3(O)\\
	\dmpr{6}-\dm{6} &=27(O) + 27\div_0(u) - 2\div_0(t)-3\div_\infty(t) &
	&= 3(O)+15\div_0(u)
	\\
	\dmpr{7}-\dm{7} & & & = 3(O) \\
	\dmpr{9}-\dm{9} &= 243(O) -20\div_0(t) & & = 3(O) \\
	\dmpr{12}-\dm{12} &= 
	\cdots & & = 
	3(O) + 15\div_0(u).
	\end{align*}
	The difference is indeed only in the valuations lying
	over the places $t=0$ and $t=\infty$ of additive reduction of $E$.
\end{example}

The following lemma shows how much the primitive valuations
of the sequence~$(\dnpr)_n$ can be ``postponed'' to later terms of $(\dn)_n$. 

\begin{lemma}\label{lem:additivemv}
	Let $K$, $F$, $E$, $P$, $v$, $P'$ and $v'$ be as in Notation~\ref{not:Fprime}.
	Suppose that $P$ has infinite order and that $E$ has
	additive reduction at~$v$.
	Let $d=d_v$ be the order of $P$ in 
	the component group
	$E(F_v)/E_0(F_v)$.
	
	Let $m = m(P,v)$ and $m'=m(P',v')$ be the ranks
	of apparition of the valuations $v$ and $v'$
	in the elliptic divisibility sequences associated to $P$ and~$P'$.
	\begin{enumerate}
	\item
	\label{it:additivemv1} 
	We have $m'\mid d$.
	\item\label{it:additivemv2} If $K$ has characteristic $0$, then $m = d$ or $m(v)=\infty$.
	\item\label{it:additivemv3} If $K$ has characteristic $p>0$, then $m = d$ or $m = dp$.
	\end{enumerate}
\end{lemma}
\begin{proof}
	Parts \eqref{it:additivemv2} and~\eqref{it:additivemv3}
	are exactly part \eqref{it:additive1} of Lemma~\ref{lem:additive}.
    
    It remains to prove \eqref{it:additivemv1}.
	After base-changing to $F'_{v'}$,
	we get a Weierstrass equation $\cstE$ over $K$.
	As $\cstE$ is defined over~$K$, it has good reduction
	at~$v'$.
	We have an isomorphism $\phi : E_{F'}\rightarrow \cstEFpr$.
	Claim: $\phi(E_0(F_{v}))\subset (\cstEFpr)_1(F'_{v'})$.
	Assuming the claim, we get $dP'=\phi(dP)\in (\cstEFpr)_1(F'_{v'})$,
	hence $v(\dmpr{d})>0$, hence $m'\mid d$.
	So in order to prove \eqref{it:additivemv1},
	it suffices to prove the claim.
	
	Proof of the claim.
	By \cite[VII.1.3(d)]{Silverman_arithmetic}
	there are $u$, $r$, $s$, $t\in\mathcal{O}_{v'}$ with $u\not=0$
	such that for all $Q=(x,y)\in E(F'_{v'})$ and $(x',y')=\phi(Q)$:
	\begin{align*}x &= u^2x' + r,\qquad\mbox{and}\qquad y = u^3 y' + u^2sx' + t.
	\end{align*}
	In fact, we have $v'(u)>0$ as otherwise
	$E$ has good reduction already with its model over~$F_v$.
	
	It now suffices to show that for points $Q$ of good reduction
	(i.e., inside $E_0(F_{v})$),
	we have $x(Q)\not\equiv r$ modulo $v$.
	Using a translation of the coordinates $x$ and $y$ of $E$
	by the elements $r$ and $t$ of $\mathcal{O}_{v'}$ we may
	assume without loss of generality that $r=t=0$
	(but now $E$ is given by a non-minimal Weierstrass equation
	over $F'_{v'}$ and $Q\in E(F'_{v'})$).
	As we have $v'(u)>0$, we find from \cite[Table~III.1.2]{Silverman_arithmetic}
	that
	$a_1 \equiv -2s$, $a_2 \equiv s^2$,
    $a_3 \equiv a_4\equiv a_6\equiv 0$,
    so the reduction of our model of $E$ modulo $v'$ is $y^2 - 2sxy = x^3 +s^2x^2$,
    The only point with $x=0$ is the singular point $(0,0)$,
    so $x(Q)\not\equiv 0$ modulo $v'$.
    This proves the claim.
\end{proof}

\begin{example}\label{ex:sscontinuedagain}
	In Example \ref{ex:sscontinued}
	the valuations of $F$ at which $E$ has additive
	reduction are $t=0$ and $t=\infty$, corresponding
	respectively to $O$ and $\div_0(u)$ of $C'$.

	
	The reduction at $t=0$ is of type $II$, hence the component
	group there has order~$2$.
	As the point $P$ does not reduce to the singular point,
	we have $d=1$.
	In the sequence, we see $m=m(P,v) = 3 = p$ and $m'=m(P',v')=1=d$.
	
	The reduction at $t=\infty$ is of type $III^*$, hence
	the component group has order~$2$.
	As the point $P$ reduces to the singular point, we have
	$d=2$.
	In the sequence, we see $m=m(P,v) = 6 = d p$ and $m'=m(P',v')=2=d$.
	
	In both cases, we have $m'=m(P',v')\leq 2 \leq 4$.
\end{example}

\begin{proposition}\label{prop:hasvaluation}
	Let $F$, $E$, $P$ be as in Notation~\ref{not:Fprime}.
	Suppose that $P$ has infinite order and that $(E,P)$ is non-constant.
	Let $n$ be a positive integer.
	If $E$ is supersingular, assume that $\mathrm{char}(F)\nmid n$.
	Then
	\begin{enumerate}\item $\dn$ has a primitive valuation \emph{or}
		\item there is a valuation
	$v$ of $F$ such that $n$ divides the order $d_v$ of
	$P$ in the component group $E(F_v)/E_0(F_v)$.	
	\end{enumerate}	
\end{proposition}
\begin{proof}
	Let $v'$ be a primitive valuation of $\dnpr$,
	which exists by Theorem~\ref{thm:constE}.
	Let $v$ be the restriction of $v'$ to $F$.
	Then $E$ has good or additive reduction at $v$ by Lemma~\ref{lem:additive}\eqref{it:constant}.
	If $E$ has good reduction, then $n=m(P',v')=m(P,v)$ by Lemma~\ref{lem:constj}.
	If $E$ has additive reduction, then $n=m(P',v')\mid d_v$
	by Lemma~\ref{lem:additivemv}\eqref{it:additivemv1}.
\end{proof}

As we have $d_v\leq 4$ by Lemma~\ref{lem:additive}\eqref{it:additive2},
we get the following result.
\begin{theorem}\label{thm:constjweak}
	Let $E$ and $P$ be as in Notation~\ref{not:Fprime}.
	Suppose that $E$ is ordinary, that $P$ has infinite order,
	and that $(E,P)$ is non-constant.
	Then for all $n>4$,
	the term $\dn$ has a primitive valuation.\qed
\end{theorem}

\begin{proposition}\label{prop:hasnovaluation}
	Let $E$ be a supersingular elliptic curve over $F$.
	Let $P\in E(F)$ be a point of infinite order.
	Suppose that $(E,P)$ is non-constant.
	Let $n$ be a positive integer.
	Then
	\begin{enumerate}\item $\dm{np}$ has no primitive valuation \emph{or}
		\item there is a valuation
	$v$ of $F$ such that $n$ divides the order $d_v$ of
	$P$ in the component group $E(F_v)/E_0(F_v)$.
	\end{enumerate}
\end{proposition}
\begin{proof}
	If $\dm{np}$ has no primitive valuation, then we are done.
	Otherwise, let $v$ be such a primitive valuation, so $m(P,v) = np$.
	Let $v'$ be an extension of $v$ to~$F'$.
	Then $E$ has good or additive reduction at $v$ by Lemma~\ref{lem:additive}\eqref{it:constant}.
    If $E$ has good reduction, then $m(P',v')=m(P,v)=np$ by Lemma~\ref{lem:constj},
    but that contradicts Theorem~\ref{thm:constE}.
    If $E$ has additive reduction, then Lemma~\ref{lem:additivemv}\eqref{it:additivemv3}
    gives
    $np = m(P,v) \mid d_vp$, so $n\mid d_v$.
\end{proof}

As we have $d_v\leq 4$ by Lemma~\ref{lem:additive}\eqref{it:additive2},
Propositions \ref{prop:hasvaluation} and~\ref{prop:hasnovaluation}
imply the following result.
\begin{theorem}\label{thm:supersingular}
Let $F$, $E$, $P$ be as in Notation~\ref{not:Fprime}.
Suppose that $E$ is supersingular, that $P$ has infinite order,
and that $(E,P)$ is non-constant.
Let $p$ be the characteristic of~$F$.
Then
\begin{enumerate}
	\item for all integers $n>4$ with $p\nmid n$, the term
	$D_n$ has a primitive valuation, and
	\item for all integers $n>4$, the term
	$D_{pn}$ has \emph{no} primitive valuation.\qed
\end{enumerate}
\end{theorem}

\section{Component groups}\label{sec:component_groups}

In order to sharpen Theorems \ref{thm:constjweak} and~\ref{thm:supersingular} further,
we need to look at the component group.
In this section we derive extra restrictions on
the order $d_v$ of a point in the component group.

By a \emph{local function field}, we mean a completion $K(C)_v$ of
the function field $K(C)$ of a smooth, projective, geometrically irreducible
curve $C$ over a field~$K$ at a discrete valuation~$v$ with $v(K)=\{0\}$ and $v(F)\not=\{0\}$. 
\begin{proposition}\label{prop:char24notinS}
	Let $F$ be a local function field of characteristic~$2$
	with valuation $v$ and constant field~$K$.
	Let $E$ be an elliptic curve over $F$ with $j(E)\in K^*$.
	Then the component group $E(F)/E_0(F)$ does not have an element of order~$4$.
\end{proposition}
\begin{proof}
	Suppose that the component group
	$E(F)/E_{0}(F)$ has an element of order~$4$.
	We will show $v(j(E))\neq0$, which contradicts our assumption that $j(E)$
	is a non-zero constant.
	
	By the tables of reduction types in \cite{Szydlo_reduction_types, Silverman_book2}
	(see the detailed references
	in the proof of Lemma~\ref{lem:additive}\eqref{it:additive1} above),
	if $E(F)/E_{0}(F)$ has an  element of order~$4$, then
	the elliptic curve $E$ has reduction at $v$ of type $I_{n}^*$
	for some $n=2m+1$ with $m\geq 0$.
	By Szydlo~\cite[Table~7]{Szydlo_reduction_types} (see also Theorems 5.1 and 6.1 of loc.~cit.),
	it follows that $E$ has a $v$-minimal Weierstrass model with
	\begin{equation}\label{eq:valuations}
	v(a_1)\geq 1,\quad v(a_2)= 1,\quad v(a_3)=m+2,\quad v(a_4)\geq m+3,\quad v(a_6)\geq 2m+4.
	\end{equation}
	As an alternative reference: under the assumption that $K$
	is perfect, one can also obtain \eqref{eq:valuations}
	from Dokchitser and Dokchitser \cite[Proposition~2]{Dokchitsers_Tate},
	using the fact that (in characteristic~$2$) $b_6=a_3^2$.
	
	The $j$-invariant equals
	\[j(E)= \frac{a_1^{12}}{\underbrace{a_1^4(a_2 a_3^2 + a_1 a_3 a_4 + a_4^2 + a_1^2 a_6)}_{\alpha}+\underbrace{a_1^3 a_3^3}_{\beta} +\underbrace{a_3^4}_{\gamma} }.\]
	Let $r=v(a_1)$, so $r\geq 1$.
	We find $v(\alpha) = v(a_1^4a_2^{\phantom{2}}a_3^2) = 4r+2m+5$
	as all other terms in $\alpha$ have larger valuation.
	
	Write
	$m=2r-1+A$ for some $A\in\ZZ$.
	It follows that
	\[v(\alpha)=8r+2A+3,\quad v(\beta)=9r+3A+3, \quad v(\gamma)=8r+4A+4.\]
	If $A\geq 0$, then
	\[v(\alpha)<\min\{v(\beta),v(\gamma)\}\]
	and $v(j(E))=4r-2A-3$ is odd, hence non-zero. If $A<0$, then
	\[v(\gamma)<\min\{v(\alpha),v(\beta)\}\]
	and $v(j(E))=4(r-A-1)>0$. 
\end{proof}
\begin{proposition}\label{prop:char33notinS}
	Let $F$ be a local function field of characteristic~$3$
	with valuation~$v$ and constant field~$K$.
	Let $E$ be an elliptic curve over $F$ with $j(E)\in K^{*}$.
	Then the component group $E(F)/E_0(F)$ does not have 
	an element of order~$3$.
\end{proposition}
\begin{proof}
	Suppose that the component group $E(F)/E_{0}(F)$
	has an element of order~$3$.
	Then at the valuation $v$ the elliptic curve $E$ has reduction of type $IV$ or $IV^{*}$
	(same reference as in the proof of Proposition~\ref{prop:char24notinS}).
	
	Let $n=1$ for type $IV$ and $n=2$ for type $IV^*$.
	By \cite[Table 4]{Szydlo_reduction_types} (see also Theorems 5.1 and~6.1 of loc.~cit.),
	there exists a minimal model of the form
	$y^2=x^3+a_2 x^2+a_4 x+a_6$
	with
	\begin{equation}\label{eq:valuations2}
	v(a_2)\geq n,\quad v(a_4)\geq n+1,\quad v(a_6)=2n,\quad v(\Delta)\geq 4n.
	\end{equation}
	The $j$-invariant of $E$ is
	\[j(E)=\frac{\overbrace{2a_2^6}^{\delta}}{\underbrace{2a_2^2 a_4^2}_{\alpha}+\underbrace{a_4^3}_{\beta}+\underbrace{a_2^3 a_6}_{\gamma}}.\]
	We will show $v(j(E))\neq 0$, which contradicts our assumption that $j(E)$ is
	a non-zero constant.
	Let $m=v(a_2)-n$ and $l=v(a_4)-2n$, hence $m,l\geq 0$. It follows that $v(\delta)=6m+6n$, $v(\alpha)=2m+2l+4n+2$, $v(\beta)=3l+3n+3$ and $v(\gamma)=3m+5n$.
	
	\textbullet\ If $l\geq m$, then
	\[v(j(E)) = \left\{\begin{array}{ll}v(\delta) - v(\beta) = 3m+3>0, & \textrm{if}\ n=2,\ l=m,\\ v(\delta) - v(\gamma)=3m+n>0, & \textrm{otherwise.}\end{array}\right. \] 
	
	\textbullet\ If $l < m$, then
	\[v(j(E))=v(\delta) - v(\beta) = 6m+3n - 3l -3 >0.\qedhere\]
\end{proof}

\section{\texorpdfstring{The third term when $j=0$}{The third term when j is 0}}
	\label{sec:the_third_term_when_j_0}

In this section we give a separate result, with an elementary proof,
for the terms
$D_{3}$ and $D_{3p}$ in the case $j=0$,
because the local considerations of Section~\ref{sec:component_groups}
do not apply to that case.

We first collect some well-known results
about elliptic curves with $j$-invariant~$0$
in the following lemma,
of which we give a proof for completeness.
\begin{lemma}\label{lem:j0}
	Let $E$ be an elliptic curve with $j$-invariant $0$ over a field $L$
	of characteristic $p>0$.
	\begin{enumerate}
		\item\label{it:j01} If $p\equiv 1\ \mathrm{mod}\ 3$, then $E$ is ordinary.
		\item\label{it:j02} If $p\not\equiv 1\ \mathrm{mod}\ 3$, then $E$ is supersingular.
		\item \label{it:j0weq} If $p>3$, then $E$ has a Weierstrass model of the form
\[ y^2 = x^3 + A\]
		with $A\in L^*$.
		\item\label{it:j0final} If $p>3$ and $p\equiv 2\ \mathrm{mod}\ 3$, then any Weierstrass
		model as in \eqref{it:j0weq} satisfies
		\[[p]_{E}(x,y)=\left(A^{-\frac{p^2-1}{3}}x^{p^2}, -A^{-\frac{p^2-1}{2}}y^{p^2}\right).\]
	\end{enumerate}
	Moreover, all elliptic curves with non-zero $j$-invariant
	over fields of characteristic $2$
	and $3$ are ordinary.
\end{lemma}
\begin{proof}
	Suppose that $\cstE$ is
	an elliptic curve over $\FF_p$ with $j$-invariant~$0$.
	Then $E$ and $\cstE$ are isomorphic over~$\overline{L}$,
	and one is supersingular if and only if the other is
	(see e.g.~\cite[V.3.1(a)(i)]{Silverman_arithmetic}).
	For $p=2$ (respectively $p=3$) Example~V.4.6
	(respectively V.4.5) of loc.~cit.\ gives
	supersingular $\cstE/\FF_p$ with $j(\cstE)=0$.
	If $p>3$, then we take $\cstE : y^2 = x^3 + 1$,
	which is ordinary if and only if $p\equiv 1\ \mathrm{mod}~3$
	by 
	Example V.4.4 of~\cite{Silverman_arithmetic}.
	
	In characteristic $p>3$, there is a short Weierstrass equation
	$y^2 = x^3 + Bx+A$ and as $j(E)=0$, we get $B=0$.
	This proves~\eqref{it:j0weq}.
	
	Let $a = \sqrt[6]{A}\in\overline{F}$ and
	$\phi : E \rightarrow \cstE : (x,y)\mapsto (x/a^2, y/a^3)$,
	where again $\cstE : y^2 = x^3 + 1$.
	If $p\equiv 2\ \mathrm{mod}~3$, then we claim
	$\#\cstE(\FF_p) = p+1$. Indeed, in that case 
	the map $\FF_p\rightarrow \FF_p : x\mapsto x^3+1$ is a bijection,
	hence so is $\cstE(\FF_p)\rightarrow \P^1(\FF_p) : (x,y)\mapsto y$,
	which proves the claim.
	By \cite[Theorem 2.3.1(b) in the Second Edition]{Silverman_arithmetic},
	we then get $\Frob_p^2 + [p] = 0$ inside $\mathrm{End}(\cstE)$,
	so $[p]_{\cstE} : (x,y)\mapsto (x^{p^2}, -y^{p^2})$.
	We conclude:
	\begin{align*}[p]_{E}(x,y) &= \phi^{-1}\circ [p]_{\cstE}\circ \phi (x,y)\\ &= (a^2(x/a^2)^{p^2}, -a^3(y/a^3)^{p^2})\\
	&= (A^{-2\frac{p^2-1}{6}}x^{p^2}, -A^{-3\frac{p^2-1}{6}}y^{p^2}),
	\end{align*}
	which proves~\eqref{it:j0final}.
	
	For the final remark, it suffices to know
	that there is exactly one supersingular $j$-invariant in
	each characteristic $p\in\{2,3\}$.
	But this follows from the formula for the number of supersingular
	$j$-invariants in Corollary 12.4.6 of Katz-Mazur~\cite{Katz_Mazur}
	(that formula needs the order of the automorphism
	group of the elliptic curve with $j$-invariant zero, which
	is computed in Proposition A.1.2(c) of~\cite{Silverman_arithmetic}).
\end{proof}

\begin{proposition}\label{prop:D_3P_is_primitive_for_j_0}
	Let $K$ be a field of characteristic $p\geq 0$ with $p\not=2,3$,
	let $C$ be a smooth, projective, geometrically irreducible
	curve over $K$ and let $F = K(C)$.
	Let $E$ be an elliptic curve over $F$ with $j$-invariant $0$
	and let $P\in E(F)$ be a point of infinite order.
	
	If the pair $(E,P)$ is not constant, then
	the term $D_{3P}$ has a primitive valuation.
\end{proposition}
\begin{proof}
	As the characteristic is not $2$ or $3$ and the $j$-invariant is~$0$,
	we get a Weierstrass equation $y^2=x^3+A$ with $A\in F^*$
	(cf.~Lemma~\ref{lem:j0}\eqref{it:j0weq}).
	If $A\in (F^*)^6K^*$, then $E$ is isomorphic over $F$
	to a curve over~$K$ and the result is a special case of
	Theorem~\ref{thm:constE}.
	So we restrict to the remaining case: $A\not\in (F^*)^6K^*$.
	
	Write $P = (x_1, y_1)\in E(F)$.
	We claim that $x_1^3/A$ is non-constant.
	Indeed, suppose it is $c\in K$.
	If $c=0$, then $P$ is $3$-torsion, contradiction.
	So $c\in K^*$ and $y_1^2/A = c+1$.
	If $c=-1$, then $P$ is $2$-torsion, contradiction.
	So we get $c+1\in K^*$.
	Now compute
	\begin{align}
	A &= x_1^3 c^{-1} = y_1^2 (c+1)^{-1},\qquad\mbox{so}\\
	A &= \frac{A^3}{A^2} = \left(\frac{y_1}{x_1}\right)^6\frac{c^2}{(c+1)^3}\in (F^*)^6K^*.
	\end{align}
	Contradiction, hence $x_1^3/A$ is non-constant.
	
	As a consequence, the function $h = x_1^3/A + 4$ is also non-constant, so let
	$v$ be a valuation of $F$ with $v(h)>0$.
	We obtain $3v(x_1)-v(A)=v(h-4)=0$ and $2v(y_1)-v(A) = v(h-3)=0$, hence
	$v(A)\in 3\ZZ\cap 2\ZZ = 6\ZZ$.
	By the transformation $A\mapsto u^6A$, $x\mapsto u^2x$, $y\mapsto u^3y$,
	which does not change~$h$, we then get $v(A)=0$, hence $v(x_1)=0$.
	
	Write $x_3=x(3P)$, which we compute to be 
	\begin{equation}\label{eq:mult3j0}
	x_3 = \frac{x_{1}^{9} - 96 A x_{1}^{6} + 48 A^{2} x_{1}^{3} + 64 A^{3}}{9 x_1^2(x_{1}^{3} + 4 A)^2}.
	\end{equation}
	Recall $v(x_1)=v(A)=0$ and $v(x_1^3+4A)>0$. In particular, the valuation of the denominator
	of this expression for $x_3$ is positive.
	The numerator is congruent to $-(12A)^3$
	modulo $x_1^3+4A$, hence is $\not\equiv 0$ modulo~$v$.
	We conclude $v(x_3)<0$ and $v(x_1)= 0$ for the minimal Weierstrass equation
	$y^2 = x^3 + A$, hence $v(D_{3P})>0$ and $v(D_{P})=0$.	
\end{proof}

\begin{lemma}\label{lem:D_3pP_is_primitive_for_j_0}
    Let $K$ be a field of characteristic $p> 3$
	with $p\equiv 2~\mathrm{mod}~3$.
	Let $C$ be a smooth, projective, geometrically irreducible
	curve over $K$ and let $F = K(C)$.

	Then there exist a supersingular elliptic curve $E$ over $F$
	with $j$-invariant $0$ and a point $P\in E(F)$ of infinite order
	such that $D_{3pP}$ has a primitive valuation
	and $(E,P)$ is non-constant.
\end{lemma}
\begin{proof}
	Take any valuation $v$ and $x_1, y_1\in F$ with $v(x_1)=v(y_1)=1$.
	Let $A = y_1^2-x_1^3$, let $E:y^2 = x^3+A$
	and let $P = (x_1, y_1)\in E(F)$.
	Write $3P=(x_3,y_3)$.
	
	Note $v(A) = 2$, hence the model is minimal at~$v$.
	As $v(x_1^3)>v(A)$, the triplication formula \eqref{eq:mult3j0} gives
	$v(x_3) = 0$, so $v(D_{3P})=0$.

	As $v(x_3) = 0$ and $v(A) = 2$, the multiplication-by-$p$ formula of Lemma~\ref{lem:j0}
	gives $v(x(3pP)) = -\frac{p^2-1}{3}\cdot 2 + p^2 \cdot 0 < 0$, so
	$v(D_{3pP}) > 0$.
	As $v(x_1) = 1$ and $v(A) = 2$, the same multiplication-by-$p$ formula also
	gives $v(x(pP)) = -\frac{p^2-1}{3}\cdot 2 + p^2 >0$,
	so $v(D_{pP})=0$.
    We find that $v$ is a primitive valuation of $D_{3pP}$.
    As $v(A)=2\not\in 6\ZZ$, we find that $A$ is not a $6$th power,
    hence $E$ is not isomorphic to a curve over $K$, hence
    the pair $(E,P)$ is non-constant.
    
    Repeated use of the multiplication-by-$p$ formula gives
    that $v(x(3p^kP))$ is strictly decreasing with~$k$, hence
    $P$ is non-torsion.
\end{proof}

\begin{example}
	Let $K = \FF_5$ and $F = K(t)$.
	As in the proof of Lemma~\ref{lem:D_3pP_is_primitive_for_j_0}, take
	$P = (t,t)$ and $E : y^2 = x^3 + t^2-t^3$.
	Then
	\begin{align*}
	\dm{1} &= \dm{2}=1, &
	\dm{3} &= t + 2, \\
	\dm{4} &= t^{2} + 2 t + 4, &
	\dm{5} &= (t + 4)^{4}, \\
	\dm{6} &= (t + 1) \cdot (t + 2) \cdot (t + 3) \cdot (t^{2} + t + 2), \\
	\dm{7} &= (t^{2} + 2 t + 3) \cdot (t^{3} + t^{2} + 2) \cdot (t^{3} + 4 t^{2} + 3 t + 4), \\
	\dm{8} &= (t^{2} + 2 t + 4) \cdot (t^{4} + 2 t^{2} + 2 t + 1) \cdot (t^{4} + 3 t^{3} + 3 t^{2} + 2 t + 2), \\
	\dm{9} &= (t + 2) \cdot (t^{3} + t + 4) \cdot (t^{3} + 3 t^{2} + 4) \cdot (t^{6} + 3 t^{4} + 3 t^{3} + t + 3), &
	\dm{10} &= (t + 4)^{4} \cdot \infty^{ 12 }, \\
	\dm{15} &= (t + 4)^{4} \cdot t^{8} \cdot (t + 2)^{25},\\
	\dm{20} &= (t + 4)^{4} \cdot (t^{2} + 2 t + 4)^{25} \cdot \infty^{ 12 }.
	\end{align*}
	And indeed
	the term $D_{15P}$ has a primitive valuation~$t$.
\end{example}

\section{Additional examples}\label{sec:additional_examples}

In this section we gather examples that are crucial for the proof
of optimality in the main theorems. 
In our examples, the function field $F$ is always
$F=K(t)$ for
a field~$K$, that is, the examples have $C = \P^1$.
The following result shows that this suffices, in the sense
that the existence of such examples implies
the existence of examples over all function
fields that we consider.

\begin{theorem}
	\label{thm:smallexamplesuffices}
	Let $K$ be a field and let $F$ be the function
	field of a smooth, projective, geometrically irreducible curve over~$K$.
	Let $E$ be an elliptic curve over $K(t)$ with $j(E)\in K$ and
	let
	$P\in E(K(t))$.
	
	If there is a rational place in $\P^1(K)$ of good reduction
	of $E$, then
	there exist
	an embedding $K(t)\hookrightarrow F$,
	an elliptic curve $E'$ over~$F$, and a point $P'\in E'(F)$
	such that
	\begin{enumerate}
		\item $P'$ and $P$ have the same order,
		\item $E'$ and $E$ have the same $j$-invariant, and
		\item for every valuation $v'$ of $F$, if $v$ is the restriction
		to $K(t)$, then
		$$m(P', v') = m(P, v).$$
	\end{enumerate}
\end{theorem}

The main idea for the proof of Theorem~\ref{thm:smallexamplesuffices}
is to base change via a suitable morphism of base curves.
We will use the following results.

We denote by $\Br(f)$ the branch locus of a finite morphism $f:X\rightarrow Y$ of
normal projective curves over~$K$.
This is the image through $f$ of the set of closed points $x\in X$ for
which the map $f$ is not \'{e}tale at $x$, cf.~\cite[Definition~7.4.15]{Liu}.

\begin{proposition}
	\label{prop:transferexample}
	Let $K$ be a field. Let $C$ and $C'$ be smooth projective curves defined over~$K$.
	Let $\phi:C'\rightarrow C$ be a dominant morphism of curves over $K$.
	Let $E$ be an elliptic curve over $K(C)$ and $P$ a point in $E(K(C))$.
	Let $E'$ denote the elliptic curve obtained from the pull-back by the map
	$\phi$ and $P'$ the corresponding point on~$E'$.
	We assume that the branch locus $\Br(\phi)$ of $\phi$
	is disjoint with the set of places of bad reduction for~$E$. 
	Then for every valuation $v'$ in $K(C')$ above $v$ in $K(C)$ we have 
	\[m(P,v)=m(P',v').\]
\end{proposition}
\begin{proof}
	If $v$ is a place of good reduction for $E$ and $v'$ is any place above $v$ in $K(C')$,
	then the elliptic curve $E'$ still has good reduction at $v'$ and
	the order of the point $P'$ modulo~$v'$
	is the same as
	the order of the point $P$ modulo~$v$.

	It remains to prove the result for places of bad reduction,
	so let $v$ be such a place.
	Let $R\subset F=K(C)$ (respectively $R'\subset F'=K(C'))$
	denote the discrete valuation ring with valuation $v$ (respectively~$v'$).
	From our assumptions and \cite[Definition 7.4.15]{Liu} it follows that the extension
	$R'/R$ has ramification index~$1$ and that the corresponding extension $k'/k$
	of residue fields is separable.
	
	Let $\mathcal{E}$ be the N\'eron model of $E$ over~$R$.
	It follows from \cite[Theorem~7.2.1(ii)]{BLR_Neron_models} that the base change
	$\mathcal{E}' = \mathcal{E}\otimes_{R}R'$
	is the Néron model of $E_{F'}$ over~$R'$.
	
    Let $x$ (respectively $x'$) be the $x$-coordinate function
    of a $v$-minimal (respectively $v'$-minimal) Weierstrass equation of~$E$.
	For a point $Q\in E(F)$, we denote by $\widetilde{Q}$
    the corresponding point in~$\mathcal{E}(R)$.
    We have for every point $Q\in E(F)$ that $v(x(Q))<0$ holds
    if and only if $\widetilde{Q}$
    restricts to the zero section of the special fibre,
    that is, satisfies $\widetilde{Q}\otimes_R k = O$
    (see \cite[Corollary IV.9.2]{Silverman_book2} and \cite[Theorem 5.5]{Conrad_minimal_models}).
    By base-changing from $R$ to $R'$,
    we see that this
    happens if and only if $\widetilde{Q}'\in\mathcal{E}(R')=\mathcal{E}'(R')$
    satisfies $\widetilde{Q}'\otimes_{R'} k'=O$, hence
    if and only if $v(x'(Q))<0$ holds.

    Applying this to $Q=nP$ for any $n$, we find $v(D_{nP})>0$ if and only if $v'(D_{nP'})>0$.
    In particular, we have $m(v,P) = m(v',P')$.
\end{proof}

In order to use Proposition~\ref{prop:transferexample},
we need to find an appropriate morphism $\phi$
for every function field $F=K(C)$ and
suitable examples over $K(t)$ for prime fields
$K$.
We use the following result
to find such maps.

\begin{theorem}[\protect{Wild $p$-Belyi theorem of Katz
		\cite[Lemma~16]{Katz_Belyi}, \cite[Theorem~11]{Zapponi}}]
	\label{thm:katz_theorem_on_belyi_maps}
	Let $C$ be a smooth, projective, geometrically irreducible curve
	defined over a perfect field $K$ of positive characteristic.
	Then there exists a non-constant morphism
	$\phi: C\rightarrow\P^{1}_K$
	(over~$K$)
	that is unramified above $\AA^1$.
	\qed
\end{theorem}

\begin{proposition}
	\label{prop:mapexists}
	Let $K$ be any field.
	Let $S \subset \P^1(\overline{K})$ be a finite set and $C$ a
smooth, projective, geometrically irreducible curve over~$K$.
If $S$ does not contain $\P^1(K)$, then
there exists a non-constant morphism
$\phi : C \rightarrow \P^1_K$ (over $K$) that is
unramified above~$S$.

[Note that the hypothesis of $S$ not containing
$\P^1(K)$ is automatically satisfied if $K$ is infinite.]
\end{proposition}
\begin{proof}
	We give a proof in the case where $K$ is infinite and a proof
	in the case where $K$ is perfect. Together, these
	two proofs cover all cases.
	
	\textit{If $K$ is infinite.}
	Let $K(C)$ be the field of functions of $C$, so
	$\overline{K}\cap K(C)=K$.
	Since $C$ is smooth, it is geometrically reduced.
	As the transcendence degree of $K(C)$ is one,
	it then follows from \cite[Proposition~3.2.15]{Liu} that $K(C)$ is a
	finite separable extension of a purely transcendental extension $K(t)$ of~$K$.
	Hence there exists a separable finite morphism $f=t:C\rightarrow \P^1_K$.
	The set $\Br(f)$ is finite by \cite[Corollary~4.4.12]{Liu}.

	Write $K= \mathbf{A}^1(K)\subset \P^1(K)$ and
	let $s$ be an element in
	$K\setminus \Br(f)$,
	which exists since $K$ is infinite. We define a map
	$\eta=(x\mapsto 1/(x-s))\circ f$. It follows that $\Br(\eta)$ does not
	contain $\infty$.
	The set $\{y-x: x\in\Br(\eta), y\in K\cap S\}$
	is finite,
	so there exists an element $s'\in K$ that does not belong to it.
	The map $\phi=(x\mapsto x+s')\circ\eta$ suffices.

\textit{If $K$ is perfect.}
By Theorem \ref{thm:katz_theorem_on_belyi_maps} there exists a morphism
$f:C\rightarrow \P^{1}$ over $K$ with $\Br(f)=\{\infty\}$.
Take $s\in \P^{1}(K)\setminus S$.
There exists a fractional linear map $\alpha:\P^{1}\rightarrow\P^{1}$
over $K$ which satisfies $\alpha(\infty)=s$.
We define $\phi=\alpha\circ f$ and check that it satisfies the claim. 
\end{proof}

\begin{question}
In Proposition~\ref{prop:mapexists} we have assumed that the set $S$ is disjoint from $\P^{1}(K)$. 
In our situation this is enough for the applications, but it would be interesting to know in general whether one could drop this assumption.
We leave it here as an open question to the reader.
\end{question}

\begin{proof}[Proof of Theorem~\ref{thm:smallexamplesuffices}]
	Let $S\subset \P^1(\overline{K})$ be the set of points
	such that $E$ has bad reduction at the corresponding
	place.
	By assumption, the set $S$ does not contain $\P^1(K)$,
	so by Proposition~\ref{prop:mapexists}
	there is a morphism $\phi: C\rightarrow \P^1_K$
	that is unramified above~$S$.
	Let $E'$ (respectively~$P'$) be the base change of $E$
	(respectively $P$) to $F=K(C)$
	via~$\phi$.
	Then (1) and~(2) are clearly true,
	and (3) follows from Proposition~\ref{prop:transferexample}.
\end{proof}

In Theorem \ref{thm:smallexamplesuffices} and Proposition~\ref{prop:transferexample},
we do a change of base curve~$C$, but we do not allow a change of the base field $K$ of the base curve.
Indeed, the
following example shows that the results are false for inseparable changes of base field~$K$.
\begin{example}\label{ex:inseparable_example}
	Let $K = \FF_3(s)$, $F = K(t)$, $E : y^2 = x^3 + t^6 x + s^2$, and $P=(0,s)$. 
	The discriminant of $E$ is $-t^{18}$, hence $E$ is minimal and of good reduction
	at all places
	except $t=0,\infty$.
	At $t=0$, the model is minimal and of reduction type $Z_1$ in Szydlo's tables
	\cite[Table 4]{Szydlo_reduction_types}.
	At $t=\infty$, we have the model $Y^2 = X^3 + t^{-2} X + t^{-12}s^2$,
	which is minimal because it has discriminant $-t^{-6}$ of valuation~$6$.
	
	We get that $P$ is integral, so $\dm{1}=0$, which has no primitive valuation.
	Now take $r = -\sqrt[3]{s}\in\overline{K}$, let $K' = K(r) = \FF_3(r)$,
	and let $F'=K'(t)\supset F$.
	Take $x' = t^{-2}(x+r^2)$ and $y' = t^{-3}y$,
	so $E' : y^{\prime2} = x^{\prime3} + t^2{x'}-r^2$ is a model over~$F'$,
	hence $E$ is not minimal at $t=0$ over~$F'$.
	In fact, the model $E'$ is minimal and of reduction type $Z_1$ over~$F'$.
	
	Over $F'$, the resulting point $P'$ satisfies $x'(P') = r^2/t^2$, so $\dmpr{1}=(t)$, hence
	this term has a primitive valuation~$t=0$.
	We get $m(P,v) > 1 = m(P',v')$.
\end{example}

\subsection{General characteristic examples}
In the case of ordinary $E$ with characteristic $\not=2,3$ and $j(E)\not=0$,
we will see in Theorem \ref{thm:constj}
that every term has a primitive
valuation, except possibly $\dm{1}$ and $\dm{2}$.
The following examples show that sometimes these
two remaining terms
do not have a primitive valuation.

\begin{lemma}\label{ex:ordsharp1}
Let $K$ be a field with $p:=\charac(K)\neq 2$. Let $j\in K$ be an element,
such that if $p=3$, then $j=0$.
Then there exists an elliptic curve $E$ with $j(E)=j$ defined over
the function field $K(t)$ of $\P^{1}_{K}$ and a point $P\in E(K(t))$
of infinite order such that
\begin{enumerate}
	\item $(E,P)$ is non-constant,
	\item $E$ has at least one rational place of good reduction,
	\item \label{eq:the first assertion} $D_{P}=D_{2P}=0$,
	\item if $E$ is supersingular, then $\dm{p}$ and $\dm{2p}$
	have primitive valuations.
\end{enumerate}
\end{lemma}
\begin{proof}
	If $p=3$, take $a=1$, $b=0\in K$.
	Otherwise, let
 	$a,b\in K$ be such that
	\begin{equation}\label{eq:charnot23const}
	\cstE : y^2 = x^3 + ax + b
	\end{equation}	
	defines an elliptic curve over $K$ with $j(\cstE)=j$. 
	Let $r = t^3 + at+b\in K[t]$, which is square-free as the discriminant of
	$\cstE$ is non-zero.
	Let
	\begin{equation}
	E : y^2 = x^3 + r^2ax + r^3b,
	\end{equation}
	so $j(E) = j(\cstE) =j$.
	We find a point $P = (rt, r^2)\in E(F)$.
	Note that the given Weierstrass equation
	is minimal at all primes of $K[t]$,
	and that the point $P$ is integral at all such primes.
	Moreover, the curve $E$ has places of additive reduction of type $I_{0}^{*}$
	hence by \cite[Corollary~7.5]{Shioda_Schutt} the point $P$
	(which does not have order $1$ or $2$)
	has infinite order.
	
	The point $P'$ of Notation~\ref{not:Fprime} is
	$P' = (t, \sqrt{r})\in \cstE(K(t,\sqrt{r}))$,
	which is non-constant.
	By Lemma~\ref{lem:constconst} this proves~(1).

	Note that $K[t]$ has at most three
	primes at which $E$ has bad reduction (the roots of $r$)
	and for all fields $K$ except $\FF_2$ and $\FF_3$
	there are more than $3$
	rational points in $\AA^1(K)$, hence there is
	at least one rational place of good reduction.
	For $K=\FF_3$, our choice of $r$ has only one rational root,
	hence there are two rational affine places of good reduction.
	This proves~(2).
	
	We also find the following Weierstrass equation, which
	is minimal for the place at infinity of $K[t]$:
	\begin{equation}
	E : Y^2 = X^3 + t^{-8}r^2aX + t^{-12}r^3b,\qquad X = t^{-4}x,\ \  Y = t^{-6}y.
	\end{equation}
	Then $X(P) = t^{-4}rt$, so $P$ is also integral at that place.
	We find that $P$ is an integral point, so $\dm{1}=0$.
	
	The duplication formula gives
	\begin{equation}
	x(2P) = \frac{1}{4} (3t^2+a)^2 - 2rt,
	\end{equation}
	which is integral at all finite places of $K[t]$.
	We also get
	\begin{equation}
	X(2P) = \frac{1}{4} (3+at^{-2})^2 - 2rt^{-3},
	\end{equation}
	which is integral at infinity.
	We find that $2P$ is an integral point, so $\dm{2}=0$.
	This proves~\eqref{eq:the first assertion}.
	
	Now suppose that $E$ is supersingular.
	Then $j\in\FF_{p^2}$, so we take $a,b\in \FF_{p^2}$ from the beginning.
	If $p=3$, then we moreover have $j=0$ and we take $a=1$, $b=0$.
	It remains only to prove that $\dm{mp}$
	does have primitive valuations for $m=1$, $2$.
	
	We have $P' = (t, \sqrt{r})\in\cstE(K(t,\sqrt{r}))$ and
	the valuation $\infty$ that appears in $\dmpr{1}$ with multiplicity
	$1$ does not appear in $\dm{1}$.
	
	Next, we claim $[p] = \psi\circ \mathrm{Frob}_{p^2}$ on $\widetilde{E}$
	with $\psi : (x,y)\mapsto (u^2x,u^3y)$
	for some $u\in K^*$.
	If $p>3$, then the claim follows from \cite[Corollary II.2.12]{Silverman_arithmetic},
    which applies as $\FF_{p^2}$ is perfect and $a,b\in \FF_{p^2}$.
	In case $p=3$, we have $\cstE : y^2 = x^3 + x$ over $\FF_3$ 
	and a direct calculation proves $[3] = \psi\circ \mathrm{Frob}_{9}$
	with $u=-1$, cf. \cite[Theorem 2.3.1(b) in the Second Edition]{Silverman_arithmetic}.
	
	We conclude that the valuation $\infty$ appears with multiplicity $p^2$ in $\dmpr{p}$,
	hence appears in $\dm{p}$
	with multiplicity
	$p^2 - v_\infty(t^{-4}r)\geq p^2-8-v_{\infty}(r)> -v_{\infty}(r)>0$,
	hence $m(\infty) = p$.
	
	The valuations $v$ at the roots of $r$, which appear in $\dmpr{2}$ do not appear in $\dm{2}$.
	They also do not appear in $\dmpr{p}$
	(otherwise by the strong divisibility property \eqref{eq:strongdivproperty}
	and $\gcd(2,p)=1$ they would appear in $\dmpr{1}=0$),
	hence they do not appear in $\dm{p}$ either.
	They do appear with multiplicity $p^2$ in $\dmpr{2p}$, hence appear in $\dm{2p}$
	with multiplicity at least $p^2-v(r) =p^2-1>0$, thus $m(v)=2p$.
\end{proof}

\begin{example}[Ordinary]\label{ex:ordsharp2}
	In Lemma~\ref{ex:ordsharp1}, take $K = \FF_5$, $a=b=1$,
	so $j(E)=1$.
	We obtain
	\begin{align*}
	\dm{1} &= \dm{2}=1, &
	\dm{3} &= (t + 3) \cdot (t + 4) \cdot (t^{2} + 3 t + 4),
	\\
	\dm{4} &= (t^{3} + 2 t^{2} + 4 t + 4) \cdot (t^{3} + 3 t^{2} + 4), &
	\dm{5} &= (t^{2} + 2 t + 4)^{5} \cdot \infty^{ 2 }
	\end{align*}
	where $\dm{1}$ and $\dm{2}$ are trivial, as we already saw
	in Lemma~\ref{ex:ordsharp1}.
\end{example}
\subsection{Examples in characteristic 3}
\begin{example}[Ordinary]\label{ex:char_3_j_ord}
	Let $K$ be a field of characteristic $3$ and let $j\in K^*$.
	We consider the elliptic curve
	\[E_{0}: y^2= x^3+j^2 x^2+2 j^5 \]
	with $j$-invariant~$j$.
	We consider the quadratic twist $E_{0}^{(d)}$ of the curve $E_{0}$ over $K(t)$ where $d= t^3+j^2 t^2+2j^5 $. 
	The curve $E_{0}^{(d)}: y^2=x^3+j^2 d x^2+2 j^5 d^3$ is non-constant and has $j$-invariant $j$ and discriminant $j^{11}d^6$.


	This is a generic fibre of a Kummer K3 surface with
	places of bad reduction only at the roots of $d=0$ and at $t=\infty$,
	all of type $I_{0}^{*}$
	(by e.g.~\cite[Table~4]{Szydlo_reduction_types}).
	On the curve $E_{0}^{(d)}$ we have a point $P=(t\cdot d, d^2)$ of height $1$ (hence non-torsion cf. \cite{Shioda_Mordell_Weil}) which satisfies the condition $D_{P}=D_{2P}=1$, since $x(2P)=t^4 + 2 j^5 t + j^7$. 
\end{example}

\begin{example}[Supersingular]\label{ex:char33}
	Let $K$ be a field of characteristic $3$.
	We consider the curve
    \[E_{t}: y^2=x^3+t^3x+t^4\]
    over $K(t)$,
    which has a point $P=(0,t^2)$.
    The discriminant of the equation $E_{t}$ is $2t^9$,
    hence there is no place of bad reduction away from $0,\infty$.
    By \cite[\S 5 and Table 4]{Szydlo_reduction_types} the reduction type at $t=0$ is $IV^{*}$ and at $t=\infty$ is $III$ and our model is minimal at all places.
    By Shioda's height formula \cite[Theorem 8.6]{Shioda_Mordell_Weil} the point $P$ has height $1/6$ hence is non-torsion.
    A direct computation of the divisors $D_{nP}$ reveals
	\[\dm{1} = \dm{2}=\dm{3}=1,\quad
	\dm{4} = t + 2,\quad
	\dm{5} = t^{2} + t + 2,\quad
	\dm{6} = \infty^{ 2 },\quad
	\dm{9} = t^{6},\quad
	\dm{27} = t^{60}.
	\]
\end{example}
\begin{remark}
	Here is how we came up with the curve and point in Example~\ref{ex:char33}.
	We wanted a pair $(E,P)$ such that $j(E)=0$,
	$\charac K=3$, and $P$ is a point of infinite order such that
	$D_{P}=1$, $D_{3P}=1$, and $D_{9P}$ has a primitive valuation~$v$.
	Such an elliptic curve $E$ has a Weierstrass model $y^2=x^3+Ax+B$
	with $A,B\in K(C)$.
	We look for a valuation $v$ of bad additive reduction for $E$
	such that the group of components has order~$3$,
	that is,
	reduction of type $IV$ or $IV^{*}$ at~$v$
	(see the proof of Lemma~\ref{lem:additive}).
	Moreover, the point $P$ should intersect a non-trivial component
	at $v$ and the point $3P$ should intersect the component of the
	zero section but should not be zero itself.
	Automatically, by additive reduction in characteristic~$3$,
	the point $9P$ then hits the zero section at~$v$.
	
	From \cite{Oguiso_Shioda} it follows that there are only two possible structures for the N\'{e}ron-Severi group of a rational elliptic surface $\mathcal{E}\rightarrow\mathbf{P}^{1}$ over an algebraically closed field of any characteristic which admit a primitive embedding of the lattice $E_{6}$ (which corresponds to the reduction type $IV^{*})$, namely $U\oplus E_{6}\oplus A_{1}\oplus \langle 1/6\rangle$ (type $49$) and $U\oplus E_{6}\oplus A_{2}^{*}$ (type $27$). Over the complex numbers both types of the N\'{e}ron-Severi group exist,
	cf.~\cite{Perrson_Configurations}.
	
	An example of such an elliptic surface with type $49$ over an algebraically closed field of characteristic $3$ was constructed in \cite[4.2.18, case 6A, 5.]{JLRRSP}.
	The generic fibre over $\mathbf{F}_3(t)$ of that surface is $E_{49, t}$
	where for $s\in \mathbf{F}_3(t)$, we define
	$$E_{49, s} : y^2=x^3+s^3(s+2) x+s^4(s^2+s+1).$$
    It has reduction of type $III$ at $t=-2$,
	reduction type $IV^{*}$ at $t=0$ and no other singular fibres. 

	 It is easy to verify that $E_{t}$ from Example~\ref{ex:char33}
	 is isomorphic over $\mathbf{F}_{3}(t)$ to $E_{49,\frac{t}{2+t}}$.
\end{remark}

\subsection{Examples in characteristic 2}

\begin{example}[Supersingular]\label{ex:char2_j_0_order_8}
	We consider a rational elliptic surface with Weierstrass equation:
	\[y^2 + ty = x^3 + t^2x\]
	over $K(t)$ for any field $K$ of characteristic~$2$.
	We have that $j(E)=0$ so the curve is supersingular. The equation above has discriminant $t^4$, hence there is no bad reduction away from $0,\infty$.
	It has bad additive reduction at $t=0$ (type $IV$) and at $t=\infty$ (type $I_{1}^{*}$)
	over $K(t)$
	by the extended Tate algorithm in~\cite{Szydlo_reduction_types} (Table 5 for $t=0$ and Table 7 for $t=\infty$ with the model $y^2+t^2y=x^3+tx^2$).
	From the Oguiso-Shioda classification \cite{Oguiso_Shioda}
	it follows that the rank of the group
	$E(\overline{K}(t))$
	is $1$
	and the group is freely generated by a point of height $1/12$.
	We checked that the point $P = (t,0)$ satisfies this condition.
	
	It is easy to verify that the divisors $D_{P}$, $D_{2P}$, $D_{3P}$ and $D_{4P}$ are trivial and $D_{6P}$ is supported at $t=0$ and $D_{8P}$ is supported at $t=\infty$. More precisely,
	\begin{align*}
	\dm{1} &= \dm{2}=\dm{3}=\dm{4}=1,&
	\dm{5} &= t + 1, \\
	\dm{6} &= t,&
	\dm{7} &= t^{2} + t + 1,\\
	\dm{8} &= \infty^{ 2 },&
    \dm{9} &= t^{3} + t^{2} + 1,\\
    \dm{10} &= (t + 1)^{4},&
    \dm{11} &= (t^{5} + t^{4} + t^{3} + t^{2} + 1),\\
    \dm{12} &= t^{5},&
     \dm{13} &= (t^{3} + t + 1) \cdot (t^{4} + t + 1),\\
    \dm{14} &= (t^{2} + t + 1)^{4},&
    \dm{15} &= (t + 1) \cdot (t^{8} + t^{7} + t^{3} + t + 1),\\
    \dm{16} &= \infty^{ 10 }, &
   \dm{17} &= (t^{4} + t^{3} + t^{2} + t + 1) \cdot (t^{8} + t^{7} + t^{6} + t^{5} + t^{4} + t^{3} + 1),\\
    \dm{18} &= t \cdot (t^{3} + t^{2} + 1)^{4},&
    \dm{19} &= (t^{6} + t^{4} + t^{3} + t + 1) \cdot (t^{9} + t^{6} + t^{4} + t^{3} + 1),\\
    \dm{20} &= (t + 1)^{16}.
    \end{align*}
\end{example}

\begin{example}[Supersingular]
	\label{ex:char2_j0_2and4}
 Let $E_{k}: y^2+t^{2k} y = x^3+ t^2 (t+1) x^2+ tx$, $k\geq 1$ be an elliptic curve over $K(t)$ for any field $K$ of characteristic $2$. 
 The curve $E_k$ has discriminant $t^{8k}$ and no bad reduction away from $0,\infty$. 
 We apply the extended Tate algorithm \cite{Szydlo_reduction_types} to places $t=0$ and $t=\infty$. 
 For $t=0$ our model is minimal for each $k$ and of type $III$.
 For $k=1$ the model of $E_{k}$ with $s=1/t$
 $$y^2+s^4 y =x^3+(s+s^2)x^2+s^7x$$
 is minimal at $s=0$ ($t=\infty$) and of reduction type $I_5^{*}$ by the extended Tate algorithm and Table 7,
 cf.~\cite{Szydlo_reduction_types}.
 For $k=2$ the model of $E_{k}$ with $s=1/t$
 $$y^2+s^2y=x^3+(s+s^2)x^2+s^7 x$$
 is minimal at $s=0$ and of type $I_1^{*}$. 

 There exists a point $P=(0,0)$ on $E_{k}$ which is not of order $2$ or $4$,
 hence it is of infinite order on this curve by \cite[Corollary~7.5]{Shioda_Schutt}.
 \begin{itemize}
 \item[(a)] If $k=1$, then $D_{2P}$ and $D_{4P}$ have a primitive valuation. More precisely,
\[ \dm{1} = 1,\quad \dm{2}= t,\quad \dm{3}= (t^2 + t + 1)\cdot (t^3 + t + 1),\quad\dm{4} = t^6 \cdot \infty^2.  \]
\item[(b)]If $k=2$, then $D_{2P}$ and $D_{8P}$ have a primitive valuation. More precisely,
\[ \dm{1} = 1,\quad 
\dm{2} = t^3,\quad
\dm{3} = t^9 + t^8 + 1,\quad
\dm{4} = t^{16},\quad 
\dm{6} = t^3 \cdot (t^9 + t^8 + 1)^4,\quad 
\dm{8} = t^{68} \cdot \infty^2.\]
\end{itemize}
\end{example}

\begin{example}[Ordinary]\label{ex:char_2_j_not_0}
	Let $K$ be a field of characteristic $2$ and $j\in K^{*}$.
	For any $a\in K(t)\setminus K$ we have an elliptic curve
	\[E_{a}: y^2+xy= x^3+(a+\frac{1}{a^2 j})x^2+\frac{1}{j}\]
	with a point $P=(a,0)$.
	Let $a=t$.
	Then $E_{t}$ is a generic fibre of an elliptic K3 surface with bad reduction at $t=0$
	and $t=\infty$.
	If $j$ is a square in $K$, then we have type $I_{4}^{*}$ at $t=0$ and
	otherwise this is type $K_{8}$ according to \cite[\S 5.1, \S 5.2]{Szydlo_reduction_types}.
	In both cases the model
	\[E_{min}: (y')^2 + t x' y' = (x')^3 + (t^3 + \frac{1}{j})(x')^2 + t^6\frac{1}{j}\]
	obtained via a transformation $x=1/t^2 x'$, $y=1/t^3 y'$ is minimal at $t=0$
	(see also \cite[6.12]{Szydlo_phd} with the model obtained from $E_{min}$ by
	mapping $x\mapsto x+t^3$).

	There is a model at $t=\infty$, of the form (with respect to $t=1/s$)
	\[E_{inf}: (y'')^2 + s x'' y'' = (x'')^3 + (\frac{1}{j}s^4 + s) (x'')^2 + \frac{1}{j} s^6.\]
	It is minimal and of type $I_{4}^{*}$ if $j$ is a square in $K$
	and of type $T_3$ if $j$ is not a square in~$K$,
	cf.~\cite{Szydlo_reduction_types} or \cite[6.14]{Szydlo_phd}.
	 The point $(t,0)$ is not a $2$-torsion point,
	 hence it is of infinite order by \cite[Corollary~7.5]{Shioda_Schutt}.
	 The point $P$ in the model $E_{min}$ has the form
	 $P_{min}=(t^3,0)$ and the point $2P_{min}$ on $E_{min}$ satisfies the condition $x(2P_{min})=t^4 + 1/j$, so the points are integral and integral at infinity, hence the divisors $D_{P}$ and $D_{2P}$ have empty support. 
	\longcomment{F0<jsq>:=FunctionField(GF(2));
		F<t>:=FunctionField(F0);
		j:=jsq^2;
		a:=t;
		E:=EllipticCurve([1,a+1/(a^2*j),0,0,1/j]);
		Emin:=EllipticCurve([t,t*(t^2+1/jsq),0,t^4/jsq,0]);
		//this is minimal by Dokchitser^2 Tate algorithm (Prop. 2)
		iso:=Isomorphism(E,Emin);
		assert BettiNumber(Emin,2) eq 22; //(hence K3, so transformation for the model at infinity 
		//with x(1/t)*t^4
		P:=E![a,0];
		iso(P)[1];
		iso(2*P)[1];
		Evaluate(iso(P)[1],1/t)*t^4;
		Evaluate(iso(2*P)[1],1/t)*t^4;}
\end{example}

\section{Proof of the main theorems}\label{sec:proof_of_the_main_theorem}

We now have all the ingredients required for proving the
following two main theorems.
\begin{theorem}\label{thm:constj}
\thmconstjcontents
\end{theorem}
\begin{proof}
	For the first assertion, by Theorem~\ref{thm:constjweak},
	it suffices to prove that $D_{3P}$ and $D_{4P}$ each have a primitive
	valuation.
	Let $p$ be the characteristic of~$K$.
	
	\textit{Proof that $D_{3P}$ has a primitive valuation.}
	Recall that $E$ is ordinary.
  By Proposition~\ref{prop:hasvaluation},
in order to show that $D_{3P}$ has a primitive valuation,
it suffices to show that for every valuation $v$
of~$F$, the order $d_v$ of $P$ in the
component group $E(F_v)/E_0(F_v)$ is not~$3$.

If $j(E)\not=0$ and $p\not=3$, then
Lemma~\ref{lem:additive}\eqref{it:constb} gives $d_v\not=3$.
If $j(E)\not=0$ and $p=3$, then
Proposition~\ref{prop:char33notinS}
gives $d_v\not=3$.

	If $j(E)=0$, then $p\not=2$, $3$ by Lemma~\ref{lem:j0}\eqref{it:j02},
	so in that case
	$D_{3P}$ has a primitive valuation by Proposition~\ref{prop:D_3P_is_primitive_for_j_0}.
	
		\textit{Proof that $D_{4P}$ has a primitive valuation.}
		Again by Proposition~\ref{prop:hasvaluation} it suffices
		to prove that for every valuation $v\in F$, we have $d_v\not=4$.
		If $p\not=2$, then this is Lemma~\ref{lem:additive}\eqref{it:consta}.
		If $p=2$, then this is Proposition~\ref{prop:char24notinS}.
	This proves the first assertion.

    Examples $(E,P)$ where the terms $\dm{1}$ and $\dm{2}$
    also have a primitive valuation are trivial to find:
    just start from an arbitrary pair $(E,Q)$ and take $P=3Q$.
    
    It remains to find examples $(E,P)$ for every field~$F=K(C)$
    and every ordinary $j\in K$
    where the terms $\dm{1}$ and $\dm{2}$ do not have
    primitive
    valuations.
    
    By Theorem~\ref{thm:smallexamplesuffices},
    it suffices to find such examples $(E,P)$ for each rational function field
    $F=K(t)$, where $K$ ranges over
    all fields,
    such that $E$ has good reduction at at least
    one place of degree one in $\P^1(K)$.
    
    For $K$ of characteristic not $2$ or $3$,
    and any ordinary $j$-invariant $j\in K$, 
    Lemma~\ref{ex:ordsharp1}
    does the trick. Note that the example has at most three
    affine places of bad reduction and there are more than $3$
    rational affine places in $\P^1(K)$, hence at least one rational place of good reduction.
    We obtain $\dm{1}=\dm{2}=0$, hence no primitive valuations.
    
    Suppose that $K$ has characteristic $2$ or $3$ and that $j\in K$
    is an ordinary $j$-invariant. Then $j\not=0$, so $j\in K^*$.
    For $K$ of characteristic~$3$, we have Example~\ref{ex:char_3_j_ord}
    for any $j\in K^*$.
    Then $d(0)\in K^*$, hence $E$
    has good reduction at the affine rational place $t=0$.
    We obtain $\dm{1}=\dm{2}=0$, hence no primitive valuations.
    \longcomment{
    K:=GF(3);
    F<t>:=FunctionField(K);
    j := 1;
    d := t^3+j^2*t^2+2*j^5;
    E:=EllipticCurve([0,j^2*d,0,0,2*j^5*d^3]);
    LocalInformation(E);                                                         
    [ <(t^3 + t^2 + 2), 6, 2, 4, I0*, true>, <(1/t), 6, 2, 1, I0*, true> ]
    Factorization(Numerator(Discriminant(E)));
    [
    <t^3 + t^2 + 2, 6>
    ]
    j:=2;
    d := t^3+j^2*t^2+2*j^5;
    E:=EllipticCurve([0,j^2*d,0,0,2*j^5*d^3]);
    LocalInformation(E);
    [ <(t^2 + 2*t + 2), 6, 2, 4, I0*, true>, <(t + 2), 6, 2, 2, I0*, true>, <(1/t), 
    6, 2, 2, I0*, true> ]
}

    For $K$ of characteristic $2$, we have Example~\ref{ex:char_2_j_not_0}
    for any $j\in K^*$.
It has good reduction at $t=1$.
We obtain $\dm{1}=\dm{2}=0$, hence no primitive valuations.
\end{proof}

\renewcommand{\yesentry}[1]{\!\!\!\begin{array}{c}\yesmark\\ #1\vphantom{,}\end{array}\!\!\!}
\renewcommand{\noentry}[1]{\!\begin{array}{c}\nomark\\ #1\vphantom{,}\end{array}\!}
\renewcommand{\bothentry}[2]{\!\!\!\begin{array}{c}\bothmark\\ #1, #2\end{array}\!\!\!}

\begin{table}[htb]
	\maintablecontents
	\caption{Table referred to in Theorem~\ref{thm:supersingularsharp}
		and its proof}
	\label{tab:table2}
\end{table}
\begin{theorem}\label{thm:supersingularsharp}
	\theoremsupersingularsharpcontents{\ref{tab:table2}}
\end{theorem}
\begin{proof}
	For each entry, the letter(s) below it
	refer(s) to one or more of the proofs listed below.
	In case of~$\bothmark$, the letters before the comma refer
	to examples where the term has a primitive
	valuation, and the letters after the comma to examples where it does not.
	If multiple letters are given, then each separately gives
	a complete proof.
	
	By Proposition~\ref{prop:hasvaluation},
	in order to
    prove that $\dn$ has a primitive valuation for $p\nmid n$,
	it suffices to prove for every additive valuation $v$ of $F$
	that $n$ does not divide the order $d_v$
	of $P$ in the component group $E(F_v)/E_0(F_v)$.
	
	By Proposition~\ref{prop:hasnovaluation},
	in order to prove that $\dm{np}$ has no primitive valuation,
	it also suffices to prove for every additive valuation $v$ of $F$
	that $n$ does not divide the order $d_v$
	of $P$ in the component group $E(F_v)/E_0(F_v)$.
	\begin{enumerate}[$A$.]
		\item Lemma~\ref{lem:additive}\eqref{it:additive2} states $d_v\leq 4$,
		and if $p\not=2$, then Lemma~\ref{lem:additive}\eqref{it:consta}
		states $d_v\not=4$.
		\item Take any pair $(E, Q)$ with $Q\in E(F)$ of infinite order.
		Then for $P=5Q$, the pair $(E,P)$ is such an example.
		To see this, apply the result in the final two columns
		(or Theorem~\ref{thm:supersingular}) to $(E,Q)$.
		\item Here $p>3$. If $j(E)\not=0$, then $d_v\not=3$ by
		        Lemma~\ref{lem:additive}\eqref{it:constb}.
		        If $j(E)=0$, then $\dm{3}$ has a primitive valuation
		        by Proposition~\ref{prop:D_3P_is_primitive_for_j_0}.
		\item Here $p\equiv 1~\mathrm{mod}~3$, so $j(E)\not=0$
		        by Lemma~\ref{lem:j0}\eqref{it:j01}.
		        But then $d_v\not=3$ by
		        Lemma~\ref{lem:additive}\eqref{it:constb}.
	    \item \label{it:D3pP} This is Lemma~\ref{lem:D_3pP_is_primitive_for_j_0}.
	\end{enumerate}
    To prove the cases with
    $\bothmark$, by Theorem~\ref{thm:smallexamplesuffices},
it suffices to find examples $(E,P)$ for each rational function field
$F=K(t)$ (over every field $K$ of the appropriate characteristic)
such that $E$ has good reduction at at least
one place of degree one in $\P^1(K)$.
The following are such examples.
\begin{enumerate}[$A$.]
	\setcounter{enumi}{5}
		\item \label{it:ordsharp1} Lemma~\ref{ex:ordsharp1}
		         gives examples for all characteristics
		         $p\geq 3$ where $\dm{1}$ and $\dm{2}$ do not have
		         primitive valuations, and $\dm{p}$ and $\dm{2p}$
		         do. They have good reduction at at least one place.
		\item \label{it:char2_j0_2and4}
		In Example~\ref{ex:char2_j0_2and4}(a)
		the terms $\dm{2}$ and $\dm{4}$ have primitive valuations.
		In Example~\ref{ex:char2_j0_2and4}(b)
		the terms $\dm{2}$ and $\dm{8}$ have primitive valuations.
		These examples are supersingular
		over $\FF_2(t)$ and have good reduction at $t=1$.
		\item \label{it:char2j08} 
		Example~\ref{ex:char2_j_0_order_8}
		gives a supersingular elliptic curve and point
		in characteristic~$2$, where $\dn$ has a primitive
		valuation for $n=6$ and $n=8$, but not for
		$n\leq 4$.
		It has good reduction at the rational place $t=1$.		
  		\longcomment{
			> K:=GF(2);
			> P<t>:=FunctionField(K);
			> E:=EllipticCurve([0,0,t,t^2,t+1]);                             
			> jInvariant(E);
			0
			> LocalInformation(E);
			[ <(t), 4, 2, 3, IV, true>, <(1/t), 8, 3, 4, I1*, true> ]
		}
	\item\label{it:char33}
	Example~\ref{ex:char33} gives a supersingular elliptic curve
	and point in characteristic~$3$ such that
	$\dn$ has a primitive valuation for $n=6$ and $n=9$,
	but not for $n\leq 3$.
	It has good reduction at the rational place $t=1$.\qedhere
	\longcomment{> K:=GF(3);
		> P<t>:=FunctionField(K);
		> E:=EllipticCurve([0,0,0,t^3,t^4]);
		> jInvariant(E);
		0
		> LocalInformation(E);
		[ <(t), 9, 3, 3, IV*, true>, <(1/t), 3, 2, 2, III, true> ]
	}
	\end{enumerate}
\end{proof}

\end{document}